\documentclass[a4paper]{amsart}
\usepackage[textwidth=31pc,textheight=51.5pc]{geometry}
\usepackage{tikz}
\usepackage{amsmath}
\usepackage{cases}
\usepackage{graphicx}
\usepackage{caption}
\usepackage{subcaption}
\usepackage{url}
\usepackage{hyperref}
\usepackage{xcolor}
\usepackage{bm}

\newtheorem{theorem}{Theorem}[section]
\newtheorem{lemma}[theorem]{Lemma}

\theoremstyle{definition}

\newtheorem{example}[theorem]{Example}

\theoremstyle{remark}
\newtheorem{remark}[theorem]{Remark}

\numberwithin{equation}{section}

\begin{document}

\title[Extrapolation of Stationary Random Fields Via Level Sets]{Extrapolation of Stationary Random Fields Via Level Sets}

\author{Abhinav Das}
\address{Kalaiya Sub-Metropolitan City, Ward No:. 7, District:. Bara, Province:. 2,  Nepal}
\email{abhinabdas7@gmail.com}

\author{Vitalii Makogin}
\address{Institute of Stochastics,
University of Ulm, 
Germany}
\email{vitalii.makogin@uni-ulm.de}

\author{Evgeny Spodarev}
\address{Institute of Stochastics, 
University of Ulm, 
Germany}
\email{evgeny.spodarev@uni-ulm.de}


\date{\today}



\begin{abstract}
In this paper, we use the concept of excursion sets for the extrapolation of stationary random fields. Doing so, we  define excursion sets for the field and its linear predictor, and then minimize the expected volume of the symmetric difference of these sets under the condition that the univariate distributions  of the predictor and of the field itself coincide.  We illustrate the new approach on Gaussian  random fields.
\end{abstract}

\maketitle


\section{Introduction}\label{sect:Intro}
 In geostatistics,  a Gaussian random field is one of standard
models for the regionalized variable $X$. For such fields, {\it kriging} is an appropriate extrapolation technique.   Here, {\it simple  kriging } (with a known mean of the field) coincides with the Gaussian {\it linear regression}, see e.g. \cite[p. 302]{shi96}.    Various kriging methods are also widely used for the extrapolation of stationary random fields with a finite second moment. They yield the best linear unbiased predictor. The optimality criterion is given here by the smallest mean square error of the estimation.  Depending on the assumptions about $X$,  several types of kriging are available, cf. e.g. \cite{chil1999geostatistics,Spodarev_2014,Stein99,wackernagel2013multivariate}. In the finite variance case, the non--linear regression (where the predictor is the conditional expectation of the regionalized variable  provided observations of the field) is still applicable, although hard to compute. Beyond different deterministic extrapolation methods (such as e.g. triangulation \cite{KletRoz04,LaiSchu07}, splines \cite{ANZH14-19,Holhorn13}, radial extrapolation \cite{Bianc17}, or reproducing kernel Hilbert spaces \cite{BerlTho04,ScheuSchaSchla13}), it is also important to mention  the classical spectral $L^2$--theory of linear prediction of stationary processes, cf.  \cite{Roz67}. 

The literature on the  inter- or extrapolation as well as prediction of random processes and fields is huge (see e.g. \cite{Cressie93,DigRib07,GaeGu10,Mather19,Scheu09,ScheuSchaSchla13,Schlaetal15} and references therein, to mention just a few). In the infinite variance case, however, the approaches are tailored to specific classes of processes or fields under consideration (such as e.g. $\alpha$-stable \cite{KSS11,Moh17,Moh09,SamTaq94,Spodarev_2014}).   However, the general framework for the extrapolation  of  heavy-tailed random fields  is still missing. 

We try to fill this gap  { by noting that two random fields are, in a sense, similar if their level  (or excursion)  sets are similar.  To be more precise,  two random fields  modeling some feature with the same structure of excursions have the same total amount of this feature exceeding each level over a fixed time interval or a spatial domain. This is certainly of interest for practical applications to insurance (with the feature being the claim size), environmetrics (e.g. for the amount of environmental pollution or radiation), etc. } In our approach proposed below,  similarity is measured by the expected volume of the symmetric difference of the level sets. It is sometimes also called {\it expected distance in measure}.  Although other measures of similarity such as e.g. the mean Hausdorff distance are also thinkable, our choice is motivated by the relatively simple structure and computational tractability of the appearing mean error terms. We are looking for a linear predictor of the field values which minimizes this expected volume cumulated over a finite number of chosen excursion levels. To enforce the uniqueness of extrapolation, a constraint is added that the univariate distributions of the field and of the predictor coincide.  { To motivate this constraint, recall that conditional simulation (see e.g. \cite{Lantu02}) is a popular alternative to extrapolation which mimics the conditional distribution of the field provided the observations are set. Its clear drawback is however usually quite long run times due to  extensive computation methods such as  Markov Chain Monte Carlo.  Our constraint  enables us  to use the advantage of conditional simulation (equality in distribution for marginals) without being necessarily computationally demanding. }  In the case of linear predictors, this constraint seems very natural within the class of infinitely divisible random fields.  { Indeed, the class of possible predictors is rich enough there (and in some sense similar to kriging).  However, also for general stationary random fields, the class of linear predictors satisfying the above constraint is not empty, since it contains at least all predictors that are equal to observed values of the field.} 

Excursions of random fields are known to describe the geometry and the extremal behaviour of sufficiently smooth random surfaces pretty well, see e.g. \cite{AdTay07,AzWsche09,Tomita90}. More recently, extrapolation and Bayesian analysis were used for level set estimation in the Gaussian setting \cite{Azzetal16,AzzGins18,BolLind15,Chevetal14,VazMar06}.  
 
 The paper is organized as follows: after introducing some notation, 
 the very general extrapolation approach for stationary measurable infinitely divisible random fields without any integrability assumptions on them is stated in Section \ref{sect:ExtrLevSet}. Its use is illustrated in Section \ref{sect:GaussRF} for stationary Gaussian random fields. There, our extrapolation problem appears to be a well--known linear programming problem with linear as well as quadratic constraints, a special case of the Second Order Cone Programming. Its complete solution is presented including the issues of existence and uniqueness. The solution is different depending on whether the mean of the field is assumed to be unknown or zero which shows direct parallels to ordinary or simple kiriging.  It is shown that the new extrapolation method is exact. Moreover, it differs from the ordinary or simple kriging. Its consistency is investigated as well.
 Section \ref{sect:NumSim} provides a numerical simulation study showing that the new extrapolation performs well in Gaussian processes case. 
 
 {Since the volume of excursion sets is a Lebesgue integral of the corresponding indicator function, replacing a  random field by its indicators enables us to extrapolate also non--Gaussian random fields without any moment or tail conditions.}
 We apply our approach to heavy--tailed  infinitely divisible random fields (such as $\alpha$--stable) in forthcoming papers.

Introduce some notation.
Let $\langle \cdot,\cdot \rangle$ be the Euclidean scalar product  and $\| \cdot \|$  the Euclidean norm in  $\mathbb{R}^n$. We write ${\bf e}=(1,1,1,\ldots,1)^\top \in\mathbb{R}^n$ for the vector with all coordinates equal to one. Let $v_{n}(B)$ denote the volume of a measurable set $B \subset \mathbb{R}^n$, and $\mathbb{I}$\{C\} be the indicator function of a  set $C$.
Moreover, we us the standard notation $$\Phi_{\mu,\sigma} \left(x\right) = \frac{1}{\sigma\sqrt{2 \pi}} \int_{-\infty}^{x} \exp\left(-\frac{(y-\mu)^2}{2\sigma^2 }\right)dy, \quad \overline{\Phi}_{\mu,\sigma} \left(x\right)=1-\Phi_{\mu,\sigma} \left(x\right),\quad x \in \mathbb{R}$$
  for the c.d.f. of $N\left(\mu,\sigma^{2} \right)$--law, whereas
we write $\Phi,\overline{\Phi}$ for $\Phi_{0,1}$ and $\overline{\Phi}_{0,1},$ respectively.  

 
 \section{Extrapolation via Level Sets}\label{sect:ExtrLevSet}

As pointed out before, there is no unified theory yet for the extrapolation of (possibly heavy-tailed) random fields. The goal of this section is to propose such a framework which compares the volumes of level sets of the field itself and of its extrapolator.

On a complete probability space $(\Omega,\mathcal{F} ,\mathbf{P}),$ consider a real-valued (strictly) stationary  measurable  infinitely divisible random field $X =\ \left\{X(t), t\in \mathbb{R}^d \right\}$  with marginal distribution function $F_{X}:$ $F_{X}(x) = \mathbf{P}\left(X\left(0\right) \leq x\right),$ $x \in \mathbb{R}$. Let $\left\{X(t_{j})\right\}^{n}_{j=1}$ be observations of the random field $X$ at locations $\{t_1,\ldots,t_n\} \subset W$ where $W$ is a non--empty compact subset of $\mathbb{R}^d$. We would like to estimate the value $X(t)$ at a location $t \notin \{t_1,\ldots,t_n\}$. Assume that  $\widehat{X}(t)$ is a linear extrapolator of the random field $X$ such that
    \begin{equation}\label{eq:Extrapolator}
        \widehat{X}(t) =\ \sum^{n}_{j=1}\lambda_{j}X(t_{j})
    \end{equation}
    where $\lambda_1,\ldots,\lambda_n \in \mathbb{R}$ are measurable functions of $t,t_1,\ldots,t_n$. These weights are chosen such that 
    \begin{equation} \label{a0}
        \widehat{X}(t) \stackrel{d}{=} X(0),
    \end{equation}  i.e. $F_{\widehat{X}\left(t\right)}(x) =\ \mathbf{P}(\widehat{X}(t) \leq x)=$ $F_{X}(x)$, $x \in \mathbb{R},$ $t \in W$, as well as $\widehat{X}\left(t\right)$ satisfies the minimization criterion which we are now going to introduce. Since $X$ is  infinitely divisible, the extrapolator $\{ \widehat{X}(t), t\in W\}$ belongs to the same class which makes writing the explicit constraints in \eqref{a0} meaningful and relatively easy.
    Define the \textit{excursion sets} of $\ \left\{X(t), t\in W\right\}$ and $\ \left\{\widehat{X}(t), t\in W\right\}$ for each level $u \in \mathbb{R}$ as
    \begin{equation}\nonumber
        A_{X}(u) =\ \left\{t\in W : X(t) > u\right\} \quad \mbox{and} \quad    A_{\widehat{X}}(u) =\ \left\{t\in W : \widehat{X}(t) > u\right\}.
    \end{equation}
     Since $X$ and $\widehat{X}$ are measurable, the volumes 
    \begin{equation}\nonumber
        v_{d}(A_{X}\left(u\right)) =\ \int_{W}\mathbb{I}\left\{X(t) > u\right\}dt
       \hspace{0.5cm} \mbox{and} \hspace{0.5cm} v_{d}(A_{\widehat{X}}\left(u\right)) =\ \int_{W}\mathbb{I}\left\{\widehat{X}(t) > u\right\}dt
    \end{equation}
    of the excursion sets $A_{X}$ and $A_{\widehat{X}}$ are random variables for each $u \in \mathbb{R}$.
    Consider the volume of the symmetric difference $$A_{X}\left(u\right) \Delta A_{\widehat{X}}\left(u\right) =\ \left(A_{x}\left(u\right) \setminus A_{\widehat{X}}\left(u\right)\right) \cup \left(A_{\widehat{X}}\left(u\right) \setminus A_{X}\left(u\right)\right)$$ 
    as a measure of the error which we make at level $u$ extrapolating $X$ by $\widehat{X}$.
 Fix $k$ different excursion levels $u_{j}$, $j = 1,\ldots,k$. Then the overall mean extrapolation error writes 
   $$        \sum^k_{j=1}\mathbb{E}\left[v_{d}\left(A_{X}(u_{j}) \Delta A_{\widehat{X}}(u_{j})\right)\right]. $$
   
  The extrapolator $\widehat{X}$ (or, equivalently, the choice of weight functions $\lambda_j(t)$, $j=1,\ldots,n$)  has to minimize this error  subject to a set of constraints:
   
   \begin{equation}\label{ex:ExtrProbl}  \begin{split}
    &\sum^k_{j=1}\mathbb{E}\left[v_{d}\left(A_{X}(u_{j}) \Delta A_{\widehat{X}}(u_{j})\right)\right] \hspace{0.5cm}  \longrightarrow \min_{\lambda_1,\ldots,\lambda_n}, \\
    &F_{\widehat{X}(t)}(x) =\ F_{{X}}(x), \hspace{0.5cm} x \in \mathbb{R},  \quad t\in W.   \end{split} 
    \end{equation}
  Let us simplify the target functional above. For that, we find sufficient conditions under which the minimum in \eqref{ex:ExtrProbl} can be attained.
   
   \begin{theorem}\label{thm:ExtrProb}
   {For each $t \in W$, a solution to the  problem}
    \begin{subequations}\label{a12}
    \begin{numcases}{}
    \sum^{k}_{j=1}\mathbf{P}\left(X(t) > u_{j},\widehat{X}(t) > u_{j}\right) \hspace{0.5cm}  \longrightarrow \max_{\lambda_1,\ldots,\lambda_n}\label{a3},\\
    F_{\widehat{X}(t)}(x) =\ F_{{X}}(x), \hspace{0.5cm} x \in \mathbb{R}\label{a11}
    \end{numcases}
    \end{subequations}
    {where $\lambda_1,\ldots,\lambda_n \in \mathbb{R}$ are measurable functions of $t,t_1,\ldots,t_n$,} solves also the problem \eqref{ex:ExtrProbl}.   
   \end{theorem}
   \begin{proof}
     Integrating
   \[
    \begin{split}
        \mathbb{I}\{A_{X}\left(u\right) \Delta A_{\widehat{X}}\left(u\right)\} &=\ \mathbb{I}\{A_{X}\left(u\right) \setminus A_{\widehat{X}}\left(u\right)\} + \mathbb{I}\{A_{\widehat{X}}\left(u\right) \setminus A_{X}\left(u\right)\} \\ &=\ \mathbb{I}\{X(t)>u\} + \mathbb{I}\{\widehat{X}(t)>u\} - 2\mathbb{I}\{X(t)>u\}\mathbb{I}\{\widehat{X}(t)>u\}
    \end{split} \]over $W$ and  Fubini's theorem  yield
    \begin{equation*} \label{a1}
        \begin{split}
        \sum^k_{j=1}\mathbb{E}\left[v_{d}\left(A_{X}(u_{j}) \Delta A_{\widehat{X}}(u_{j})\right)\right] &=\ \sum^{k}_{j=1}\int_{W}\mathbb{E} \mathbb{I}\{t \in A_{X}(u_{j}) \Delta A_{\widehat{X}}(u_{j}\}) \hspace{0.1cm} dt \\
        &=\ \sum^{k}_{j=1}\int_{W}\mathbf{P} \left( X(t) > u_{j} \right) + \mathbf{P}(\widehat{X}(t) > u_{j}) \\ & - 2\mathbf{P}(X(t) > u_{j},\widehat{X}(t) > u_{j}) \, dt \\
        &=\ \sum^{k}_{j=1}\int_{W}\mathbf{P} \left(X(t) > u_{j} \right)dt + \sum^{k}_{j=1}\int_{W}\mathbf{P} \left(\widehat{X}(t) > u_{j} \right) dt  \\
        & - 2\sum^{k}_{j=1}\int_{W}\mathbf{P}\left(X(t) > u_{j},\widehat{X}(t) > u_{j}\right) dt\,. \\
        \end{split}
    \end{equation*}
  By Fubini's theorem, stationarity of $X$ as well as condition ~(\ref{a0}) we get 
    \newline
    \begin{equation*}\label{a9}
    \begin{split}
        \sum^{k}_{j=1} \int_{W} \mathbf{P}(X(t) > u_{j})dt &+ \sum^{k}_{j=1} \int_{W} \mathbf{P}(\widehat{X}(t) > u_{j})dt \\
        &=\ 2\sum^{k}_{j=1} v_{d}(W)\mathbf{P}(X(0) > u_{j}) 
        =\ 2v_{d}(W)\sum^{k}_{j=1}\left(1 - F_{X}(u_{j})\right)\,.
        \end{split}
    \end{equation*}
    \hfill
    \newline
   Then the target functional in \eqref{ex:ExtrProbl}  reads
     \begin{equation}\label{a10}
    \begin{split}
        \sum^k_{j=1}\textbf{E}\left[v_{d}\left(A_{X}(u_{j}) \Delta A_{\widehat{X}}(u_{j})\right)\right] &=\ 2v_{d}(W)\sum^{k}_{j=1}\left(1 - F_{X}(u_{j})\right) \\
        &- 2\sum^{k}_{j=1}\int_{W}\mathbf{P}\left(X(t) > u_{j},\widehat{X}{\left(t\right)} > u_{j}\right)dt\,.
    \end{split}
    \end{equation}
  The {first sum on the right} does not depend on  ${\bm{\lambda}}=(\lambda_1,\ldots,\lambda_n)^\top$, so we can neglect it. Hence, minimizing  ~(\ref{a10}) w.r.t. ${\bm{\lambda}}$ means maximizing its third sum. So our optimization problem rewrites as 
    \begin{subequations}\label{a13}
    \begin{numcases}{}
    \int_{W}\sum^{k}_{j=1}\mathbf{P}\left(X(t) > u_{j},\widehat{X}(t) > u_{j}\right)dt \hspace{0.5cm}  \rightarrow \max_{{\bm{\lambda}}}, \label{a2} \\ 
    F_{\widehat{X}(t)}(x) =\ F_{{X}}(x), \quad x \in \mathbb{R}, \quad t\in W. \nonumber  \label{a8}
    \end{numcases}
    \end{subequations}
    
The functional ~(\ref{a2}) is maximal if  the sum under the integral in ~(\ref{a2}) is maximal for any $t \in W$.
Thus our final extrapolation problem reads as in \eqref{a12}.
\end{proof}    

\begin{remark}
The formulation of the extrapolation problem in Theorem \ref{thm:ExtrProb} allows for an arbitrary choice of the number $k$ and concrete numerical values  $u_1,\ldots, u_k$ of excursion levels.
Although this choice does not matter at all for Gaussian random fields (cf. Lemma \ref{a19}   below),  it may affect the performance of extrapolation for other infinitely divisible random functions. Thus, the problem of the optimal choice of  parameters $k\in \mathbb{N},$ $u_1,\ldots, u_k\in\mathbb{R}$ arises.  It is very natural  to use the mean $\mu=\mathbb{E} X(t)$ as one of  levels $u_j$ whenever the field $X$ is integrable. However, in general this problem needs a further investigation.   To avoid this discussion, one can replace the sum  in \eqref{ex:ExtrProbl}  by an integral over $\mathbb{R}$ with respect to $u$: 
  \begin{equation}\label{ex:ExtrProbl1} 
   \int_\mathbb{R} \mathbb{E}\left[v_{d}\left(A_{X}(u) \Delta A_{\widehat{X}}(u)\right)\right] du \hspace{0.5cm}  \longrightarrow \min_{\bm{\lambda}}.
        \end{equation}
To ensure the finiteness of the integral on the left handside of \eqref{ex:ExtrProbl1}, it is sufficient to require that $\mathbb{E} |X(0)| <\infty $ and that 
$$
\int_{\mathbb R}\left| {\rm Cov}\left(  \mathbb{I}\left\{X(t) \geq u\right\}, \mathbb{I}\big\{\widehat X(t) \geq u\big\}     \right)  \right| \, du 
$$
is bounded for all $t\in W.$ 
Indeed, we use Fubini theorem to write
  \begin{equation*}\label{eq:auxil}
    \begin{split}  \int_\mathbb{R} \mathbb{E}\left[v_{d}\left(A_{X}(u) \Delta A_{\widehat{X}}(u)\right)\right] du & =2   \int_{W}  \int_\mathbb{R} F_X(u) \left(1- F_X(u)\right) du\, dt \\
    -&2   \int_{W} \int_{\mathbb R}  {\rm Cov}\left(  \mathbb{I}\left\{X(t) \geq u\right\}, \mathbb{I}\big\{\widehat X(t) \geq u\big\}     \right)   \, du\, dt .
        \end{split}
    \end{equation*}
Then it is easy to see that
$$
0\le  \int_\mathbb{R} F_X(u) \left(1- F_X(u)\right) du \le \int_0^{+\infty} \left(1- F_X(u)\right) du +  \int_{-\infty}^0  F_X(u) du=\mathbb{E} |X(0)|.
$$
Using arguments from the proof of Theorem \ref{thm:ExtrProb}, we arrive at the following alternative formulation of the extrapolation problem related to \eqref{ex:ExtrProbl}:

\begin{subequations}\label{eq:diffForm}
    \begin{numcases}{}
    \int_{\mathbb R}  {\rm Cov}\left(  \mathbb{I}\left\{X(t) \geq u\right\}, \mathbb{I}\big\{\widehat X(t) \geq u\big\}     \right)   \, du \hspace{0.5cm}  \rightarrow \max_{{\bm{\lambda}}},  \label{eq:b1} \\ 
    F_{\widehat{X}(t)}(x) =\ F_{{X}}(x), \quad x \in \mathbb{R}, \quad t\in W. \nonumber  \label{eq:b2}
    \end{numcases}
    \end{subequations}
In the case of stationary Gaussian fields, both sufficient conditions are fulfilled, and the formulation \eqref{eq:b1} is equivalent to \eqref{a2}, compare the proof of Lemma \ref{a19}. 
 
\end{remark}

 So far, it is too early to speak about the existence or uniqueness of a solution to the problem \eqref{a12} in such generality. Such analysis only makes sense if the subclass of the infinitely divisible fields under consideration is specified. 
Let us illustrate our new extrapolation approach ~(\ref{a12}) by applying it to stationary Gaussian random fields. These fields, although square integrable, serve as an important benchmark model in extrapolation.  

\section{Extrapolation of Gaussian Random Fields}\label{sect:GaussRF}

Let $X=\left\{X\left(t\right), t \in \mathbb{R}^d \right\}$ be a stationary measurable Gaussian random field with mean $ \mathbb{E}X=\mu$ and covariance function 
$C\left(t\right) = \mathrm{ Cov}\left(X\left(0\right),X\left(t\right)\right) = \mathbb{E}\left[ X\left(0\right)  X\left(t\right)  \right] - \mu^2$,  $t \in \mathbb{R}^d$, 
$\sigma^2:=C\left(0\right)>0 $. For some fixed excursion levels $u_{1},\ldots,u_{k} \in \mathbb{R}$, we extrapolate $X\left(t\right)$, $t \in W \setminus  \left\{t_{1},\ldots,t_{n}\right\}$ by 
\begin{equation}
    \widehat{X}\left(t\right) = \sum_{j=1}^{n} \lambda_{j}\left(t\right) X\left(t_{j}\right),
\end{equation}
where {\bm{$\lambda$} = \bm{$\lambda$}$(t)$= $\left(\lambda_{1}\left(t\right),  \ldots, \lambda_{n}\left(t\right)\right)^{\top}$} maximizes the target functional 
\begin{equation}\label{a15}
    F\left({\bm{\lambda}},t\right) = \sum_{j=1}^{k} \mathbf{P}\left(X\left(t\right) > u_{j}, \widehat{X}\left(t\right)> u_{j}\right)
\end{equation}
 under the (set of) constraint(s) $F_{\widehat{X}(t)}(x) =\Phi_{\mu,\sigma} (x)$, $x \in \mathbb{R} .$
Since $X$ and $\widehat{X}$ have equal $N(\mu,\sigma^{2})-$ distributions, the equality of their variances  reads
\begin{equation}\label{a40}
    {\bm{\lambda}}^{\top}\Sigma{\bm{\lambda}} = \sigma^{2},
\end{equation}
where $\Sigma = \left( C\left(t_{l} - t_{j}\right) \right)_{l,j = 1}^{n}$ is the positive semidefinite covariance matrix of the vector of observations $X(t_j),$ $j=1,\ldots,n$. We will refer to ~(\ref{a40}) as to the {\it ellipsoid constraint}. If $\mu\neq 0$ is assumed to be unknown, the equality of means leads to an additional constraint $\sum_{j = 1}^{n}\lambda_{j} =1.$ In case $\mu=0$  this condition is not needed.
Introduce the notation $$ c_{t} := \left(C\left(t-t_{1}\right),C\left(t-t_{2}\right),\ldots,C\left(t-t_{n}\right)\right)^{\top}.$$

\begin{lemma}\label{a19}
Let $\mu$ be unknown. The optimization problem \eqref{a3}-\eqref{a11} for stationary Gaussian random fields rewrites as follows:
\begin{subequations}\label{a16}
\begin{numcases}{}
\left<{\bm{\lambda}}, c_{t}\right> \rightarrow \max_{{\bm{\lambda}} \in \mathbb{R}^{n}}, \\
{\bm{\lambda}}^{\top} \Sigma {\bm{\lambda}} = \sigma^{2}, \label{eq:Ellips} \\
\langle {\bm{\lambda}}, {\bf e} \rangle =1\label{eq:Simplex}
\end{numcases}
\end{subequations}
for each $t\in W$.  If $\mu=0$ then constraint \eqref{eq:Simplex} can be omitted.
\end{lemma}
\begin{proof}
By \cite[p. 9]{book} we have 
\begin{equation*}
\begin{split}
    \mathbf{P}\left(X\left(t\right)>u_{j},\widehat{X}\left(t\right)>u_{j}\right) &= \overline{\Phi}^{2}_{\mu,\sigma}\left(u_{j}\right) + \frac{1}{2\pi} \int_{0}^{\sin^{-1}\left(\rho_t\right)} \exp\left(-\frac{(u_{j}-\mu)^{2}}{\sigma^2}\frac{1-\sin\left(\theta\right)}{\cos^{2}\left(\theta\right)}\right)d\theta,
    \end{split}
\end{equation*}
where $\rho_{t} = \mathrm{Corr}\left[X\left(t \right), \widehat{X}\left(t \right)\right]$, $t \in W$. Thus, target functional \eqref{a15} rewrites
\begin{equation}\label{eq:TargetF}
     F\left({\bm{\lambda}},t\right) = \sum_{j=1}^{k} \overline{\Phi}_{\mu,\sigma}^{2}\left(u_{j}\right) + \frac{1}{2\pi} \int_{0}^{\sin^{-1}\left(\rho_t\right)}\left[\sum_{j=1}^{k}\exp\left(-\frac{(u_{j}-\mu)^{2}}{\sigma^2}\frac{1-\sin\left(\theta\right)}{\cos^{2}\left(\theta\right)}\right)\right]d\theta .
\end{equation}
To compute $\rho_{t}$, we write
$\rho_{t} = \sigma^{-2}\sum_{i=1}^{k}\lambda_{j}   \mathrm{Cov} \left[X\left(t\right),X\left(t_{j}\right)\right] := \sigma^{-2} \langle{\bm{\lambda}},c_{t}\rangle$. The function
\begin{equation*}
    g\left(\theta\right):= \sum_{j=1}^{k} \exp\left(-\frac{(u_{j}-\mu)^{2}}{\sigma^2}\frac{1-\sin\left(\theta\right)}{\cos\left(\theta\right)}\right)
\end{equation*}
is positive for all $ \theta \in \left[0,\frac{\pi}{2}\right)$. Then
\begin{equation*}
    \int_{0}^{\sin^{-1}\rho_{t}} g\left(\theta\right)d\theta \rightarrow \max_{{\bm{\lambda}}} \hspace{0.2cm} \rm{iff} \hspace{0.2cm} \sin^{-1}\rho_t \rightarrow \max_{{\bm{\lambda}}}.
\end{equation*}
As $\sin^{-1}$ is an increasing function, this is equivalent to
\begin{equation*}
    \rho_{t} \rightarrow \max_{{\bm{\lambda}}}.
\end{equation*}
We arrive at the formulation \eqref{a16}.
\end{proof}

\begin{figure}[!ht]
    \centering
    \begin{subfigure}{0.33\textwidth}
    \resizebox{\textwidth}{!}{%
    \begin{tikzpicture}
    \draw[rotate=45] ellipse(2.35 and 1.4);
    \draw[->, thick] (-3,0) -- (3,0) node (xaxis) [right] {$\lambda_{1}$} ;
    \draw[->, thick] (0,-3) -- (0,3) node (yaxis) [above] {$\lambda_{2}$};
    \draw[<-] (1.37,0.4) node [black,right] {${\bm{\lambda}}$} -- (0,0) node[black,left] {0};
    \draw[<-]  (1,2.6) node [black,right] {$c_{t}$} -- (0,0) ;
    \draw (0.09,0.2) +(15:0.01cm) arc (180:0:0.28cm);
    \draw (0,1.75) node [left]{$1$} -- (1.75,0) node [below]{$1$};
    \path[->] (0.3,0.3) ++(15:0.01cm) node{$\alpha$};
    \path[->] (3,1) ++(15:0.01cm) node{${\bm{\lambda}}^{\top}\Sigma{\bm{\lambda}} = \sigma^{2}$};
    \end{tikzpicture}
    }%
    \caption{Geometric interpretation of  SOCP problem \eqref{a16}  \label{fig:SOCPmunot0}}
    \end{subfigure}
    \hspace{1cm}
    \begin{subfigure}{0.3\textwidth}
    \resizebox{\textwidth}{!}{%
    \begin{tikzpicture}
    \draw[rotate=45] ellipse(2.35 and 1.4);
    \draw[->, thick] (-3,0) -- (3,0) node (xaxis) [right] {$\lambda_{1}$} ;
    \draw[->, thick] (0,-3) -- (0,3) node (yaxis) [above] {$\lambda_{2}$};
    \draw[<-] (1.95,0.8) node [black,right] {${\bm{\lambda}}$} -- (0,0) node[black,left] {0};
    \draw[<-]  (1,2.6) node [black,right] {$c_{t}$} -- (0,0) ;
    \draw (0.09,0.2) +(15:0.01cm) arc (180:0:0.28cm);
    \path[->] (0.3,0.3) ++(15:0.01cm) node{$\alpha$};
    \path[->] (2.6,2) ++(15:0.01cm) node{${\bm{\lambda}}^{\top}\Sigma{\bm{\lambda}} = \sigma^{2}$};
    \end{tikzpicture}
    }%
    \caption{Illustration of  SOCP problem  \eqref{a14} with $\mu=0$ \label{fig:SOCPmu0} }
    \end{subfigure}
    \label{ellipsoid}
    \caption{Geometrical interpretation of problems \eqref{a16} and \eqref{a14}}
\end{figure}
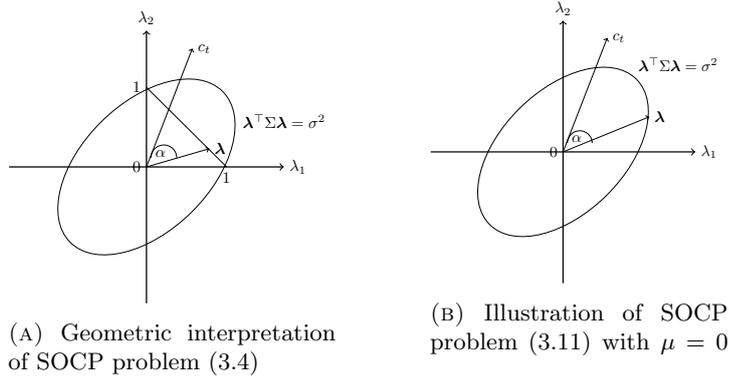

Notice that the target functional $\left<{\bm{\lambda}}, c_{t}\right> =  \| \mbox{Pr}_{c_t} {\bm{\lambda}} \| \cdot \|  c_t \| $, where $\mbox{Pr}_{y}$ is the operator of the orthogonal projection onto the line with direction vector $y\in\mathbb{R}^n$, cf. Figure~\ref{fig:SOCPmunot0}.  Lemma \ref{a19} poses our extrapolation problem as a linear programming problem with quadratic constraints which appears to be a special case of a SOCP (second order cone programming)  or a QCQP (quadratically constrained quadratic program) class, cf. e.g. \cite{AliGold03,BoydVan04}. It can be solved via the Lagrangian formalism. For that, we need the following notation.  
Introduce numbers $b_0 := c_t^{\top} \Sigma^{-1}c_t,$ $b_1 := {\bf e}^{\top} \Sigma^{-1}c_t,$ $b_2:={\bf e}^{\top}\Sigma^{-1} {\bf e}$.

\begin{theorem}\label{thm:ExactSolMuUnknown}
Let    $\mu \neq 0$ be unknown,  $c_t$ be non-collinear to ${\bf e},$ $c_t\neq \mathbf{0},$ and $\Sigma$ be  positive definite. 
Then there exists a unique solution to the problem  \eqref{a16}   which is given by
    \begin{equation} \label{a7}
     {\bm{\lambda}}= \sqrt{\frac{\sigma^2 b_2-1}{b_0b_2-b_1^2}}\Sigma^{-1}\left(c_t-\frac{b_1}{b_2}{\bf e}\right)+\frac{1}{b_2}\Sigma^{-1}{\bf e}. \end{equation}
\end{theorem}

\begin{proof}

Let ${\bf e}_j$, $j=1,\ldots,n$ be the orthonormal basis vectors of $\mathbb{R}^n$. 
Let us denote by $K$ the set of admissible vectors ${\bm{\lambda}}$ given by equations \eqref{eq:Ellips}--\eqref{eq:Simplex}. This compact set is not empty since ${\bf e}_j\in K$, $j=1,\ldots,n$. Due to positivedefiniteness of the matrix $\Sigma$,  $K$ is the boundary of an ellipsoid of dimension $n-1$ which lies in the hyperplane \eqref{eq:Simplex}. The linear functional $\left<{\bm{\lambda}}, c_{t}\right>$ is continuous and thus attains its minimum and its maximum on $K$ which are distinct one from another unless $\left<{\bm{\lambda}}, c_{t}\right>$ is constant on $K$.

We write the Lagrangian for the linear programming problem \eqref{a16} as 
\begin{equation*}
    \zeta\left({\bm{\lambda}},\gamma,\delta\right) = \langle c_{t}, {\bm{\lambda}}\rangle + \gamma\left({\bm{\lambda}}^{\top}\Sigma{\bm{\lambda}} - \sigma^{2} \right) + \delta\left(\langle {\bf e}, {\bm{\lambda}}\rangle - 1\right),
\end{equation*}
where $\gamma$, $\delta$ are Lagrange multipliers. 
Differentiating partially with respect to ${\bm{\lambda}}$, we get the gradient
\begin{equation*}\label{a20}
    \nabla_{{\bm{\lambda}}}\zeta({\bm{\lambda}},\gamma,\delta) = c_t+2\gamma \Sigma{\bm{\lambda}} +\delta {\bf e}=0.
\end{equation*}
which gives that 
\begin{equation}\label{a21}
    2\gamma{\bm{\lambda}}=-\Sigma^{-1}(c_t+\delta {\bf e}).
\end{equation}
Note that $\Sigma^{-1}$ exists and is symmetric. The left-hand side of \eqref{a21} equals ${\bf 0}$ if and only if $c_t=-\delta {\bf e}.$ Due to constraints, $\gamma  \neq 0$ iff $c_t\neq  - \delta {\bf e}.$ From condition \eqref{eq:Simplex} we get that 
$$
{\bf e}^{\top}\Sigma^{-1}(c_t+\delta {\bf e})=b_1+\delta b_2=-2\gamma, 
$$
and
\begin{equation}\label{a22}  {\bm{\lambda}}=\frac{\Sigma^{-1}(c_t+\delta {\bf e})}{{\bf e}^{\top}\Sigma^{-1}(c_t+\delta {\bf e})}
\end{equation}
assuming that $b_1+\delta b_2\neq 0.$
We find $\delta$ plugging \eqref{a22} into \eqref{eq:Ellips}  which leads to
\begin{equation*}
(c_t^{\top}+\delta {\bf e}^{\top})\Sigma^{-1}\Sigma\Sigma^{-1}(c_t+\delta {\bf e})=\sigma^2({\bf e}^{\top}\Sigma^{-1} c_t+\delta {\bf e}^{\top}\Sigma^{-1}  {\bf e})^2
\end{equation*}
and thus yields 
$$b_2(\sigma^2 b_2-1)\delta^2+2b_1( b_2\sigma ^2-1)+\sigma^2 b_1^2 - b_0=0,$$
or,  equivalently, 
\begin{equation}
\label{eq:b}
(b_2\delta+b_1)^2=\frac{b_0b_2-b_1^2}{\sigma^2 b_2-1}.
\end{equation}
Solving the minimization problem $x^{\top}\Sigma x \to \min$ subject to $x^{\top}y=1,$ we find that $y^{\top}\Sigma^{-1}y=\left(\min_{x^{\top}y=1}x^{\top}\Sigma x\right)^{-1}.$  Therefore, we have that $1/b_2\leq \frac{1}{n^2}{\bf e}^{\top}\Sigma {\bf e}<\sigma^2$ and $\sigma^2 b_2-1>0.$ 
Since $\Sigma^{-1}$ is positive definite, then 
$$0\leq (\sqrt{b_2}c_t-\sqrt{b_0}{\bf e})^{\top}\Sigma^{-1} (\sqrt{b_2}c_t-\sqrt{b_0}{\bf e})=b_2b_0-2\sqrt{b_0 b_2}b_1+b_0 b_2=2\sqrt{b_0 b_2}(\sqrt{b_0 b_2}-b_1)$$ and $b_0b_2\geq b_1^2.$ Moreover, $b_0b_2-b_1^2=0$ iff $c_t$ is parallel to ${\bf e}.$
Thus, equation \eqref{eq:b} has always solutions
\begin{equation}
\label{eq:delta12}
\delta_1=-\frac{b_1}{b_2}+\frac{1}{b_2}\sqrt{\frac{b_0b_2-b_1^2}{\sigma^2 b_2-1}}, \quad \delta_2=-\frac{b_1}{b_2}-\frac{1}{b_2}\sqrt{\frac{b_0b_2-b_1^2}{\sigma^2 b_2-1}}
\end{equation}
such that $b_1+\delta_{1,2} b_2\neq 0$ if $b_0b_2\neq b_1^2.$
The corresponding values of  $$c_t^{\top}\bm{\lambda}_{1,2}=(b_0+\delta_{1,2} b_1)/(b_1+\delta_{1,2} b_2)$$ are
$$c_t^{\top}\bm{\lambda}_{1}=\frac{b_1}{b_2}+\frac{1}{b_2}\sqrt{(b_0 b_2-b_1^2)(\sigma^2 b_2-1)},\quad c_t^{\top}\bm{\lambda}_{2}=\frac{b_1}{b_2}-\frac{1}{b_2}\sqrt{(b_0 b_2-b_1^2)(\sigma^2 b_2-1)}.$$
Obviously, $c_t^{\top}\bm{\lambda}_{1}\geq c_t^{\top}\bm{\lambda}_{2}$ and the maximizer in \eqref{a16} is
\begin{align*}
    {\bm{\lambda}}&=\sqrt{\frac{\sigma^2 b_2-1}{b_0b_2-b_1^2}}\Sigma^{-1}\left(c_t-\frac{b_1}{b_2}{\bf e}\right)+\frac{1}{b_2}\Sigma^{-1}{\bf e}=\sqrt{\sigma^2-b_2^{-1}}\frac{\Sigma^{-1}\left(c_t-\frac{b_1}{b_2}{\bf e}\right)}{\sqrt{   c_t^{\top}\Sigma^{-1} \left(c_t-\frac{b_1}{b_2}{\bf e}\right) }}+\frac{\Sigma^{-1}{\bf e}}{b_2}.
\end{align*}

\end{proof}

{By Lemma \ref{a19}, the weight vector  $\bm{\lambda}$ maximises all the probabilities \eqref{a3}  in Theorem \ref{thm:ExtrProb}, so  the problem is solved for all levels $u$ in the Gaussian case.}

\begin{remark}
\hfill
\newline
\begin{enumerate}
    \item In formula \eqref{a7}, vectors $c_t-\frac{b_1}{b_2}{\bf e}$ and ${\bf e}$ are orthogonal. Indeed, 
\begin{equation*}\begin{split}
{\bf e}^{\top}\left(c_t-\frac{b_1}{b_2}{\bf e}\right)&=b_2^{-1}(b_2 {\bf e}^{\top}c_t -n b_1)
=b_2^{-1}({\bf e}^{\top}\Sigma^{-1}{\bf e} {\bf e}^{\top}c_t -n{\bf e}^{\top}\Sigma^{-1}c_t)\\
&=b_2^{-1}(  {\bf e}^{\top}{\bf e} {\bf e}^{\top}\Sigma^{-1} {c_t}  -n {\bf e}^{\top}\Sigma^{-1}{c_t})=0,
\end{split}\end{equation*}
since symmetric matrices ${\bf e} {\bf e}^{\top}$  and $\Sigma^{-1}$ commute.
\item
If $c_t$ is parallel to ${\bf e},$ then the
maximization functional ${c_t}^{\top} {\bm{\lambda}} $ is constant under the condition ${\bf e}^{\top} {\bm{\lambda}}=1.$ Therefore, there are many solutions to the problem  \eqref{a16}, namely,  these are all vectors $\bm{\lambda}$ satisfying \eqref{eq:Ellips} and \eqref{eq:Simplex}.
For example, ${\bm{\lambda}}={\bf e}_k$, $k=1,\ldots,n$ are admissible.
\item
Under the assumptions of Theorem  \ref{thm:ExactSolMuUnknown}, the extrapolator $\widehat{X}(t)=\sum_{j=1}^n \lambda_j X(t_j)$ is exact, that is, $\widehat{X}(t_{{j}})=X(t_{{j}})$ for all ${j}=1,\ldots, n$. In this case, $c_{t_{{j}}}=\Sigma {\bf e}_{{j}}$ and $b_0={\bf e}_{{j}}^{\top}\Sigma \Sigma^{-1} \Sigma {\bf e}_{{j}}=C(t_{{j}}-t_{{j}})=\sigma^2,$ $b_1={\bf e}^{\top}\Sigma^{-1}\Sigma  {\bf e}_{{j}}=1.$ Then \eqref{a7} rewrites
$${\bm{\lambda}}= \sqrt{\frac{\sigma^2 b_2-1}{\sigma^2 b_2-1}}\Sigma^{-1}\left(c_{t_{{j}}}-\frac{1}{b_2}{\bf e}\right)+\frac{1}{b_2}\Sigma^{-1}{\bf e}=\Sigma^{-1}c_{t_{{j}}}={\bf e}_{{j}}.$$
\end{enumerate}
\end{remark}

\begin{example}\label{ex:n2}
For $n=2$, we have
$
    \widehat{X}\left(t\right) = \lambda_{1} X\left(t_{1}\right) + \lambda_{2}X\left(t_{2}\right) .  
$
Under the assumptions of Theorem \ref{thm:ExactSolMuUnknown},   
 there exists a unique vector of weights $\bm{\lambda}=(\lambda_1,\lambda_2)$ satisfying \eqref{a16}   if  $C(t-t_1) \neq C(t-t_2)$. Since the admissible ellipsoid $K$ of values 
$\bm{\lambda}$ in two dimensions is trivial, i.e.,  $K=\{ (1,0), (0,1)  \} $, we have the following solutions:
\begin{equation*}
\begin{split}
     {\bm{\lambda}}=(1,0) & \mbox{ if }  C(t-t_1)>C(t-t_2),\\
       {\bm{\lambda}}=(0,1) & \mbox{ if }  C(t-t_1)<C(t-t_2),\\
       {\bm{\lambda}}=(0,1) \mbox{ or } (1,0)   & \mbox{ if }  C(t-t_1)=C(t-t_2).\\
        \end{split}
\end{equation*}
\end{example}

Now, let us turn to the case of a centered Gaussian random field $X$, i.e. if $\mu=0$. This case is less realistic from the point of view of real applications, since the drift $\mu$ is usually unknown and has to be estimated. However, after being estimated, it can be subtracted from the field itself making it approximately centered. For $\mu=0$, the extrapolation optimization problem   in Lemma \ref{a19} can be stated without constraint \eqref{eq:Simplex} making life much more simple: for any $t\in W$
\begin{subequations}\label{a14}
\begin{numcases}{}
\left< {\bm{\lambda}}, c_{t} \right > \rightarrow \max_{\bm{\lambda} \in \mathbb{R}^{n}},   \\
{\bm{\lambda}}^{\top}\Sigma {\bm{\lambda}} =\sigma^2. 
\end{numcases}
\end{subequations}

{Since the target functional equals $\left<{\bm{\lambda}}, c_{t}\right> =  \| \mbox{Pr}_{c_t} {\bm{\lambda}} \| \cdot \|  c_t \| $, where $\mbox{Pr}_{y}$ is the operator of the orthogonal projection onto the line with direction vector $y\in\mathbb{R}^n$ (cf. Figure \ref{fig:SOCPmu0}), any point on the boundary of the ellipsoid is feasible.}
\begin{theorem}\label{thm:ExtrMMueq0}
Let $\mu = 0$. If $\Sigma$ is a positive definite matrix, and $c_t\neq \mathbf{0},$  then the above linear programming problem has a unique solution $\bm{\lambda}=(\lambda_1,\ldots,  \lambda_n)^\top $ for each $t\in W$  which reads
\begin{equation}\label{eq:ExtrMethMu0}
    {\bm{\lambda}} = \sigma\frac{ \Sigma^{-1}c_t}{\sqrt{c_t^{\top}\Sigma^{-1}c_t}}.
\end{equation}
\end{theorem}
\begin{proof}
The Lagrangian for this problem has the form 
 $$\zeta\left({\bm{\lambda}},\gamma\right) = \langle c_{t}, {\bm{\lambda}}\rangle + \gamma\left({\bm{\lambda}}^{\top}\Sigma{\bm{\lambda}} - \sigma^{2} \right),
$$
where $\gamma$ is a Lagrange multiplier.
The gradient with respect to ${\bm{\lambda}}$ equals 
$ \nabla_{{\bm{\lambda}}}\zeta({\bm{\lambda}},\gamma) = c_t+2\gamma \Sigma{\bm{\lambda}},$ which equals zero if $2\gamma {\bm{\lambda}} = - \Sigma^{-1}c_t.$ From \eqref{eq:Ellips}
we have that $\frac{1}{4\gamma^2} c_t^{\top}\Sigma^{-1}\Sigma\Sigma^{-1}c_t=\sigma^2$ which gives two extremal points
$$\bm{\lambda}_{1,2}=\pm \sigma\frac{\Sigma^{-1}c_t}{\sqrt{c_t^{\top}\Sigma^{-1}c_t}}.$$
The values of maximization functional are
$c_t^{\top}\bm{\lambda}_{1,2}=\pm \sigma \sqrt{c_t^{\top}\Sigma^{-1}c_t}.$ Obviously,  $\bm{\lambda}_1$ is its maximizing point.
\end{proof}

\begin{remark}
\hfill
\newline
\begin{enumerate}
    \item  Extrapolation method \eqref{eq:ExtrMethMu0} is exact as well. Indeed, let $c_t=c_{t_k}=\Sigma{\bf e}_k,$ then
     $c_{t}^{\top}\Sigma^{-1}c_{t}={\bf e}_k^{\top}\Sigma\Sigma^{-1}\Sigma{\bf e}_k=\sigma^2,$ and ${\bm{\lambda}}=\sigma\frac{\Sigma^{-1}\Sigma{\bf e}_k}{\sigma}={\bf e}_k.$
    \item Extrapolation methods \eqref{a7} and  \eqref{eq:ExtrMethMu0}   differ from the   ordinary or simple  kriging. There, the kriging estimator is \begin{equation}
       \label{X:kr} \tilde{X}(t)=\sum_{j=1}^n \lambda_{\mathrm{kr},j}(t)X(t_j),
    \end{equation} and
    ${\bm{\lambda}}_{\mathrm{kr}}=(\lambda_{\mathrm{kr},1},\ldots,\lambda_{\mathrm{kr},n})^{\top}=\Sigma^{-1}(c_t+\delta_{\mathrm{kr}} {\bf e})$ minimizes the functional $$\mathbf{E}(\Tilde{X}(t)-X(t))^2={\bm{\lambda}}^{\top}\Sigma {\bm{\lambda}}+\sigma^2-2c_t^{\top}{\bm{\lambda}},$$ where
     $\delta_{\mathrm{kr}}=(1-b_1)/b_2$ in the ordinary kriging case  and $\delta_{\mathrm{kr}}=0$ in the simple kriging case,  cf. \cite[p. 23,84]{wackernagel2013multivariate}, \cite[p. 155,167]{chil1999geostatistics}. Now it is sufficient to compare ${\bm{\lambda}}_{\mathrm{kr}}$ with \eqref{a7} rewritten as
     \begin{equation*} 
     {\bm{\lambda}}= \sqrt{  \frac{\sigma^2 b_2-1}{b_0b_2-b_1^2}  }\Sigma^{-1}c_t    + \frac{1-     \sqrt{\frac{\sigma^2 b_2-1}{b_0 b_1/ b_2-1}}   }{b_2} \Sigma^{-1}{\bf e}
     \end{equation*}
and with \eqref{eq:ExtrMethMu0} in the form
\begin{equation}\label{eq:ExtrMethMu0_re}
    {\bm{\lambda}} =\frac{ \sigma}{\sqrt{c_t^{\top}\Sigma^{-1}c_t}}  \Sigma^{-1}c_t.
\end{equation}  
Notice that ${\bm{\lambda}}$ in \eqref{eq:ExtrMethMu0_re} and ${\bm{\lambda}}_{\mathrm{kr}}=\Sigma^{-1}c_t$ for simple kriging are proportional.
    For instance, if $n=1$ we can compute the simple kriging estimate as
    \begin{equation*}\label{a38}
        \Tilde{X}\left(t\right) = \mathbb{E}\left(X\left(t\right) \vert X\left(t_{1}\right) \right)\\
        = \mu +    \frac{C(t-t_1)}{C(0)}  \left(X\left(t_{1}\right) - \mu\right),
    \end{equation*}
   compare e.g. \cite[Theorem 2, p. 238]{shi96}. 
   We see that $\Tilde{X}\left(t\right) \neq  \widehat{X}\left(t\right)$, 
    where $   \widehat{X}\left(t\right) = \lambda_1 (t) X\left(t_{1}\right) $ is our predictor from \eqref{a7} and  \eqref{eq:ExtrMethMu0} with
    \begin{equation*}\label{a39}
        \lambda_1(t) = \begin{cases}
        1, & \mu \mbox{ unknown},\\
        \mbox{\rm sgn} \left( C(t-t_1)\right), & \mu=0  \mbox{ known}. \end{cases}
    \end{equation*}
    \item The expected mean square error of extrapolation $\mathbb{E}\left[\widehat{X}\left(t\right)-X(t)\right]^2 =2(\sigma^2-c_t^{\top}\bm{\lambda})$ equals 
  $$\mathbb{E}\left[\widehat{X}\left(t\right)-X(t)\right]^2 =
        2 \left( \sigma^2-\frac{b_1}{b_2}-\frac{1}{b_2}\sqrt{(b_0 b_2-b_1^2)(\sigma^2 b_2-1)}  \right)
  $$
  in case of unknown $\mu$ and
    $$\mathbb{E}\left[\widehat{X}\left(t\right)-X(t)\right]^2 =     
        2\sigma\left(\sigma- \sqrt{c_t^{\top}\Sigma^{-1}c_t}\right) 
$$
for $\mu=0$.
\end{enumerate}

\end{remark}

Now we would like to discuss the consistency of our extrapolation methods. Namely, we prove that, under some additional assumptions on the covariance function $C$ of $X$, it holds $\widehat X (t) \to X(t) $ as $n\to\infty$ in mean square (and thus in stochastic) sense if the observation design $t_1,\ldots,t_n$ is asymptotically dense around the point $t$.

\begin{theorem}\label{thm:Consist}
Let the covariance function $C$ be continuous and positive definite, and $\min_{{j}=1,\ldots, n}\|t_{{j}}-t\|\to 0$ as $n\to \infty.$ 
\begin{itemize}
\item[(i)]\label{thm:Consist1} For $\mu$ either known  ($\mu=0$) or unknown, it holds
$$
\mathbb{E}\left[\widehat{X}\left(t\right)-X(t)\right]^2  \longrightarrow 0, \quad n\to\infty. 
$$
\item[(ii)]\label{thm:Consist2} Let now $t_{{j}}\in T_N= (h_N\mathbb{Z})^d\cap W$ for ${j}=1,\ldots, n(N),$ where $h_N>0$ is a mesh size and $n(N)$ is the number of points in $T_N.$
If $C$ is  H\"{o}lder continuous at 0 with index $\alpha>0$ and $\sum_{N=1}^{\infty}h_N^\alpha<\infty,$ then 
$$
\widehat{X}\left(t\right)  \longrightarrow X(t), \quad N\to\infty, \text{ a.s.} 
$$
\end{itemize}
\end{theorem}
\begin{proof} (i):
As it was mentioned in the proof of Theorem \ref{thm:ExactSolMuUnknown}, any ${\bm{\lambda}}={\bf e}_{{j}},$ ${j}=1,\ldots,n$ is admissible, i.e., ${\bf e}_{{j}}^{\top}\Sigma{\bf e}_{{j}}=\sigma^2$ and ${\bf e}_{{j}}^{\top}{\bf e}=1.$ Thus, maximum values of ${\bm{\lambda}}^{\top} c_t$ in optimization problems \eqref{a16} and \eqref{a14} are larger or equal than
${\bf e}_{{j}}^{\top} c_t=C(t-t_{{j}}),$ for any ${j}=1,\ldots,n.$ 
{Taking ${j}(n)=\mathrm{arg}\min_{l=1\ldots,n}\|t_l-t\|$, we have from the continuity of $C$ that 
}
{
$$\mathbb{E}\left[\widehat{X}\left(t\right)-X(t)\right]^2 =2(\sigma^2-{\bm{\lambda}}^{\top} c_t)\leq 2\sigma^2-2C(t-t_{j(n)})\to 0$$
}
{
as $n\to 0.$} Then the first statement is proved.


(ii): Let now $t_{{j}}\in T_N,$ ${j}=1,\ldots, n(N).$ Then  $\|t_{{j}(N)}-t\|\leq \sqrt{d} h_N/2.$ 
If, additionally, $C$ is H\"{o}lder continuous at 0 with index $\alpha$, then there exists  constants $K,D>0$ such that $|C(0)-C(x)|\leq K \|x\|^\alpha$ for all $\|x\|\leq D$ and 
$$\sum_{N=1}^{\infty} \mathbb{E}\left[\widehat{X}\left(t\right)-X(t)\right]^2 \leq 2 K \sum_{N=1}^{\infty} \|t-t_{{j}(N)}\|^{\alpha} \leq 2^{1-\alpha} d^{\alpha/2} K \sum_{N=1}^{\infty} h_N^{\alpha}<\infty.$$
Thus, it follows from the Borel--Cantelli lemma that $\widehat{X}(t)\to X(t)$  a.s. as $N\to \infty$.
\end{proof}

\begin{remark}\label{rem:DenseObservation}
It follows from the above proof that under the assumptions of Theorem \ref{thm:Consist} (ii) 
$$ \mathbb{E}\left[\widehat{X}\left(t\right)-X(t)\right]^2 \leq   2K \min_{{j}=1,\ldots, n}\|t_{{j}}-t\|^{\alpha}.$$
That is, the speed of convergence of a path of $\widehat{X}$ to that of $X$ is slower for Gaussian processes with more rough paths which is reflected by a higher constant $K$ or by a smaller value of $\alpha$. This point will be illustrated on numerical experiments in the next section.

\end{remark}

\section{Numerical Simulation}\label{sect:NumSim}

To compute solutions  \eqref{a7} and  \eqref{eq:ExtrMethMu0}  numerically, the inversion of matrix $\Sigma$ has to be replaced by the numerical solution of the corresponding system of linear equations e.g. by using the QR decomposition. Thus, one has to find $\bm{\lambda}$ out of
\begin{equation*}\label{eq:ExtrMethMuNot0_re1}
   \Sigma{\bm{\lambda}} =\sqrt{  \frac{\sigma^2 b_2-1}{b_0b_2-b_1^2}  }c_t    + \frac{1-     \sqrt{\frac{\sigma^2 b_2-1}{b_0 b_1/ b_2-1}}   }{b_2} {\bf e}, \quad \mu \mbox{  unknown},
\end{equation*}  
or
\begin{equation*}\label{eq:ExtrMethMu0_re1}
   \Sigma {\bm{\lambda}} =\frac{ \sigma}{\sqrt{c_t^{\top} \Sigma^{-1} c_t}}  c_t, \quad \mu=0 \mbox{  known},
\end{equation*}  
respectively. In statistical practice, the covariance function $C$ of the field $X$ has to be first estimated  from the data $X(t_1),\ldots, X(t_n)$. To ensure the positive definiteness of  matrix $\Sigma$, a valid covariance model with this property has to be fitted to the estimated covariance $\widehat C$ by a least squares method as it is usually done with kriging in geostatistical applications. In this section, however, we assume the covariance function $C$ to be known a priori not to bother with these well known issues. 

For numerical simulations, we choose the one-dimensional case $d=1$ due to visualization reasons. The above methods are dimension free and work for any $d>1$ in the same way. {The corresponding R code for $d=1,2,3$ in available at \cite{Rcode}.}
Let $X = \left\{X\left(t\right), t\in \mathbb{R}\right\}$ be a stationary Gaussian process with standard normal marginals, i.e. $\mu=0$, $\sigma^2=1$. We choose different covariance functions $C$, simulate the Gaussian process $X$ on the interval $[0,100]$ and measure its values at different points $T_n=(h \mathbb{Z})\cap[0,100],$ where $h$ is e mesh size. 

Thus, for the exponential covariance  function  $C_1\left(t\right) = \exp{\left(-\left\vert t \right\vert\right)}$, $t\in \mathbb{R}$  we successively observe the process $X$ at  locations $T_n$ with mesh sizes $h=10\, (n=10),$ $h=1\, (n=100),$ and $h=0.2 \,(n=500),$  respectively.
For the Gaussian covariance  $C_2\left(t\right) =  \exp{\left(-\frac{t^{2}}{2}\right)}$, $t\in \mathbb{R}$ we used $h=10\,(n=10)$ and $h=1\, (n=100).$ For the Bessel  covariance $C_3(t)= J_{0}(t) = \sum_{s = 0}^{\infty}\frac{\left(-1\right)^{s}\left(\frac{t}{2}\right)^{2s}}{s ! \Gamma \left(s+1\right)}$, $t \in \mathbb{R}$ and oscillating sine covariance $C_4(t) =\frac{\sin \left(t\right)}{t}$, $t \in \mathbb{R}$  we measure $X$ at $T_n$ with mesh sizes $h=10\, (n=10),$ $h=5\, (n=20),$ and $h=2.5 \,(n=40).$

Then we perform our extrapolation on a regular grid in $[0,100]$ with mesh size $0.1$ and compare the values  $\widehat{X}(t)$ with  $X(t)$ on that grid by taking the length of the symmetric difference of the excursion sets of $\widehat{X}$ and  $X$ at levels $u_j \in \left\{-2,-1,0,1,2\right\}$.  The results of the extrapolation of the Gaussian process $X$ with the four covariance structures as above for low observation density (based on ten measurements) are given in Figure \ref{com-small}. The exactness of the extrapolation at observation points is seen directly from the graphs of $X$ and $\widehat{X}$.

As mentioned in Remark \ref{rem:DenseObservation}, increasing the density of observations leads to better extrapolation which is controlled by the H\"older constants $K$ and $\alpha$. This can be seen in Figure \ref{com-dense} with forty observation points. There, the predictors coincide with the realisation of $X$ in the sine ($K=1/6$, $\alpha=2$) and Bessel ($K=1/4$, $\alpha=2$) case, and are very close to $X$ for the Gaussian  ($K=1$, $\alpha=2$) covariance function.  For $n=100$ observations,  the curves for $X$ and $\widehat{X}$ become indistinguishable also in the latter case.  
Figure \ref{expo-expo} illustrates the increase in resolution of the observation grid for the Gaussian process $X$ with exponential ($K=\alpha=1$)  covariance function. The paths of $X$ are more rough which is reflected by smaller $\alpha$ and larger $K$.  In accordance with Remark \ref{rem:DenseObservation}, the acceptable quality of extrapolation is reached at higher frequencies  $n=100, 500$ of observations.


\begin{figure}[!]
    \centering
    \begin{subfigure}{\textwidth}
    \includegraphics[height=0.19\textheight, width = \linewidth]{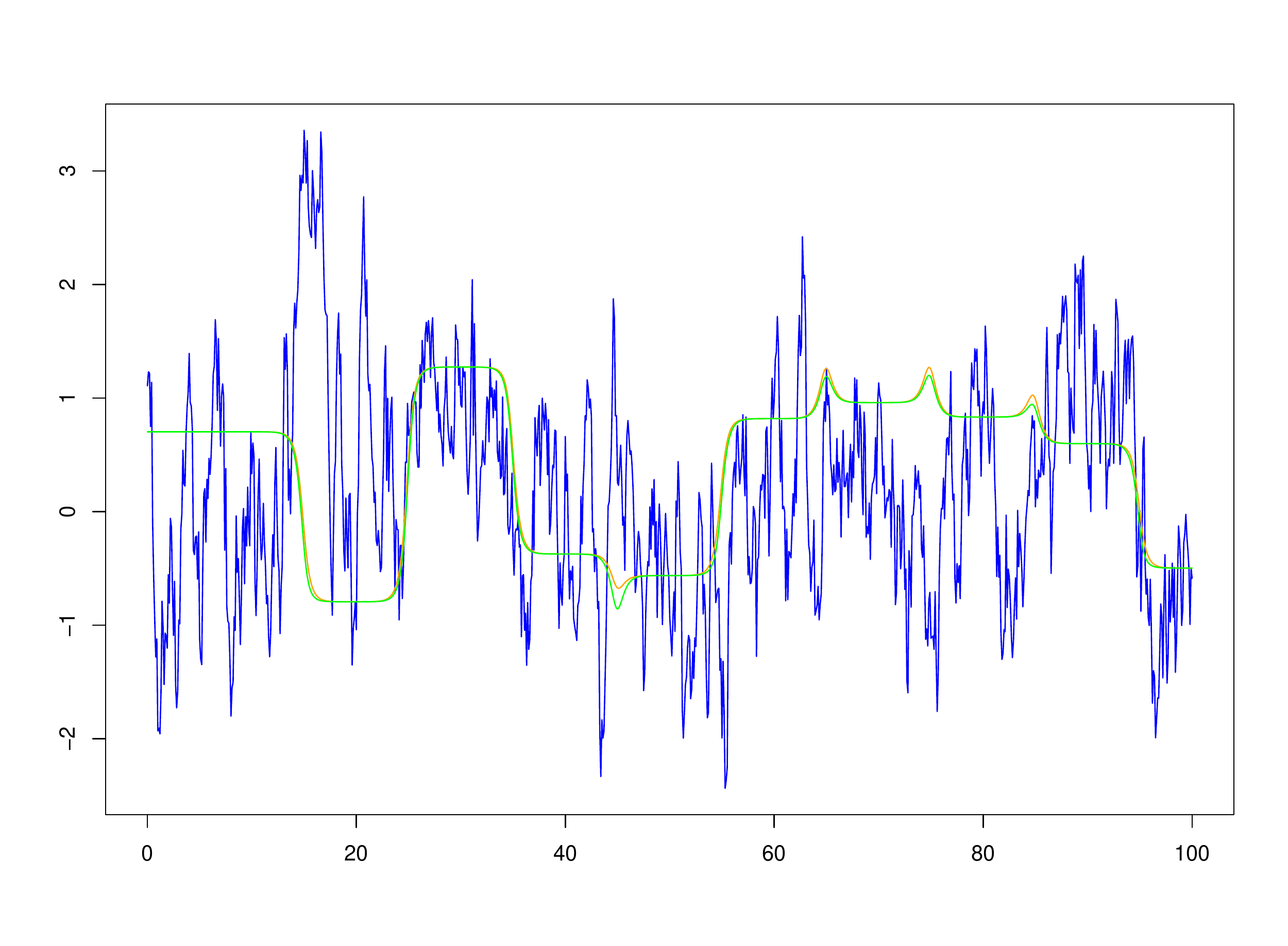}
    \caption{$X$ with exponential covariance $C_1$, observed at $t_j=10j$, $j=1,\ldots,10$.}
    \label{expo-n-10}
    \end{subfigure}
    \begin{subfigure}{\textwidth}
    \includegraphics[height=0.19\textheight, width = \linewidth]{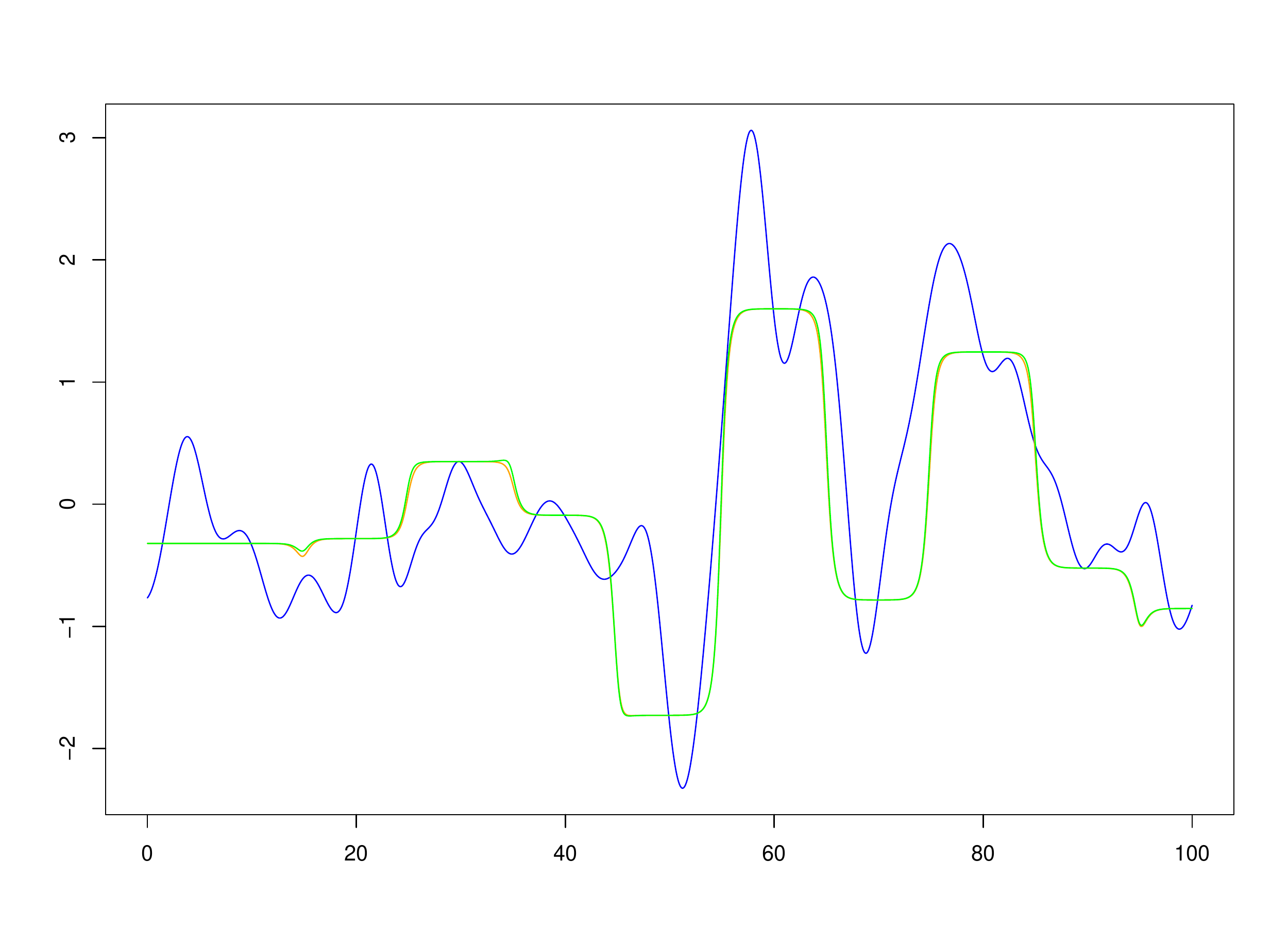}
    \caption{$X$ with Gaussian covariance $C_2$, observed at $t_j=10j$, $j=1,\ldots,10$.}
    \label{gaus-n-10}
    \end{subfigure}
    \begin{subfigure}{\textwidth}
    \includegraphics[height=0.19\textheight, width = \linewidth]{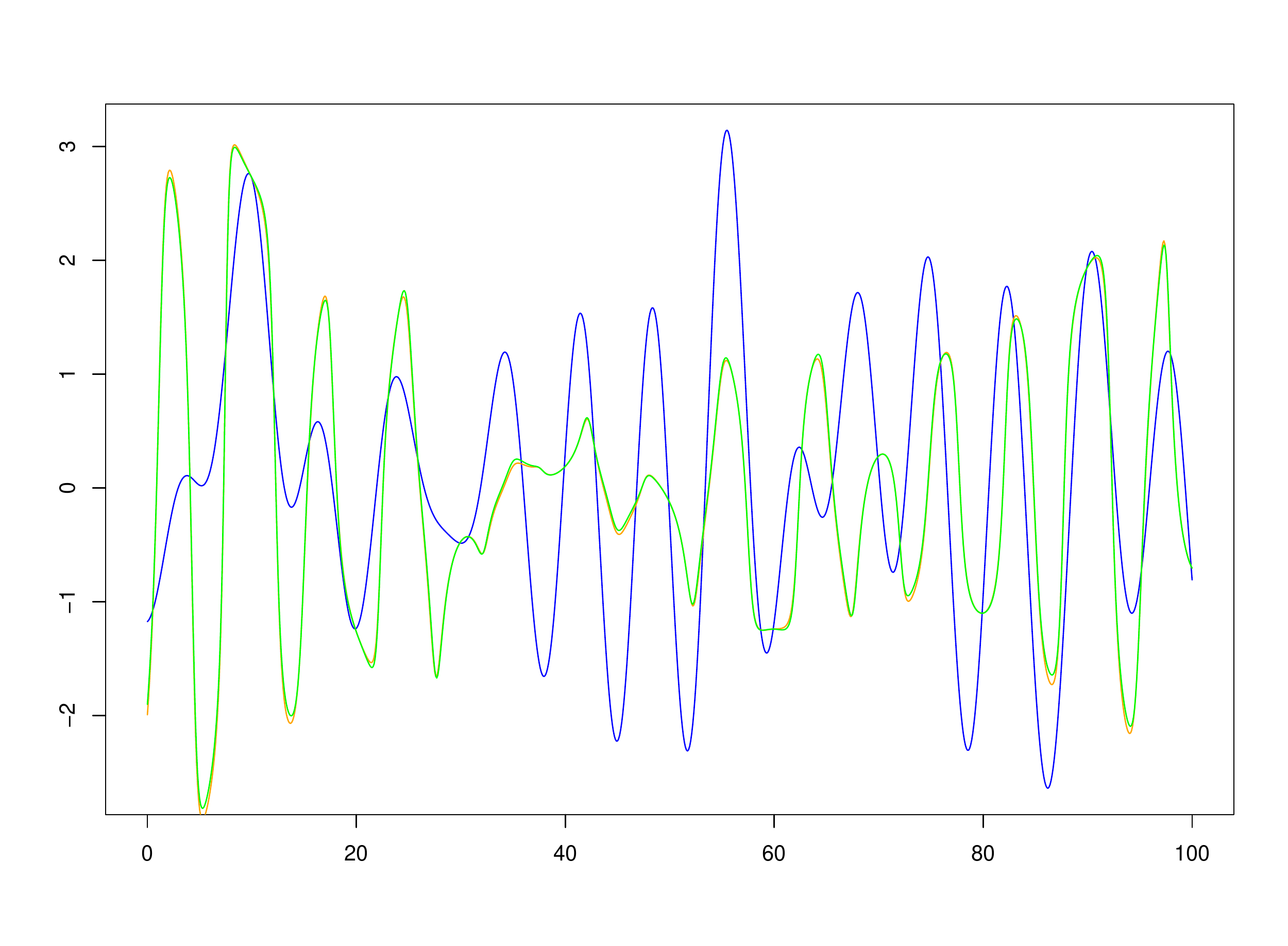}
    \caption{$X$ with Bessel covariance $C_3$, observed at $t_j=10j$, $j=1,\ldots,10$.}
    \label{bessel-n-10}
    \end{subfigure}
    \begin{subfigure}{\textwidth}
    \includegraphics[height=0.19\textheight, width = \linewidth]{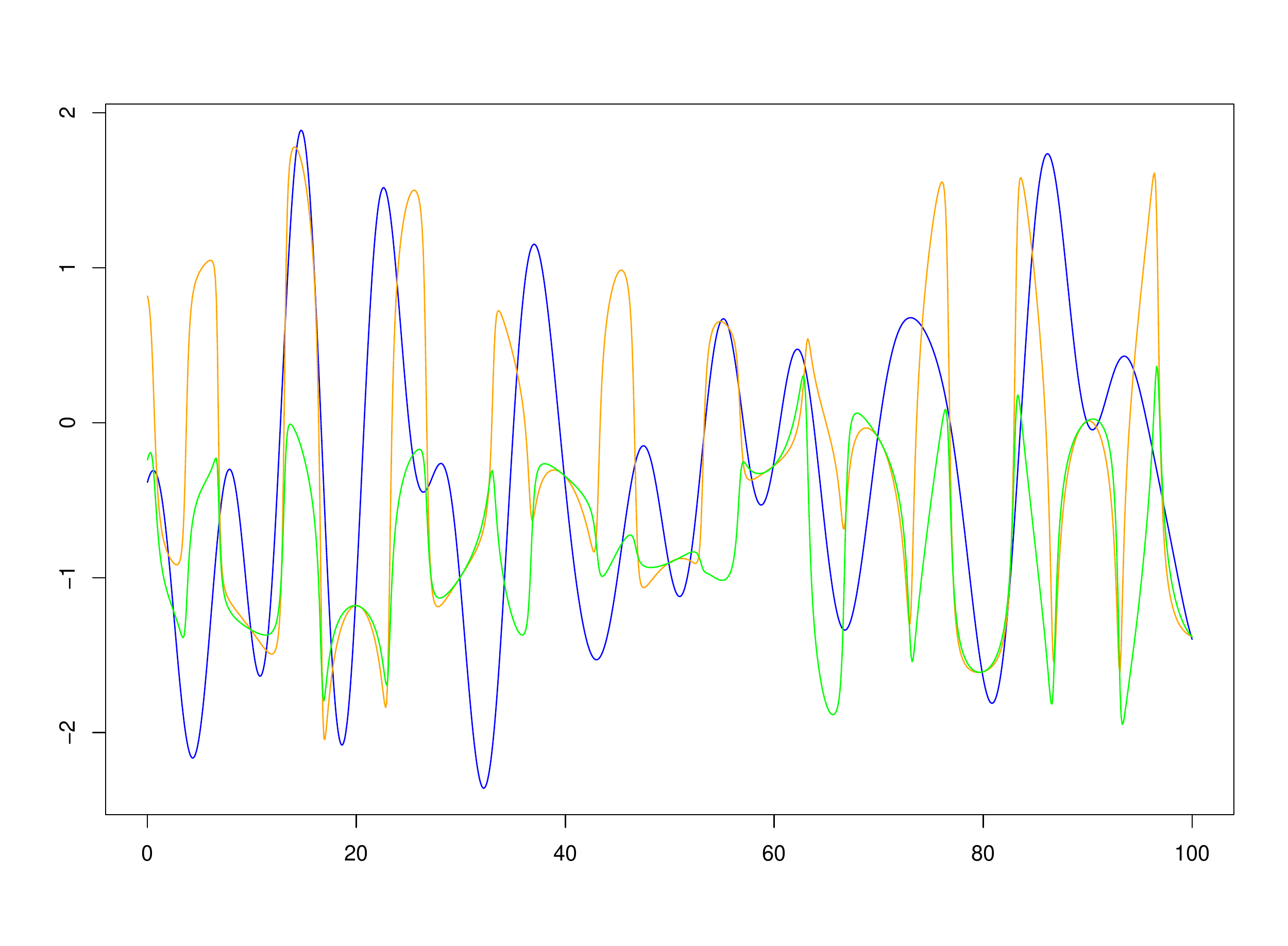}
    \caption{$X$ with oscillating sine covariance $C_4$, observed at $t_j=10j$, $j=1,\ldots,10$.}
    \label{sine-n-10}
    \end{subfigure}
    \caption{\small{Comparison of trajectories of the Gaussian process $X$ (blue) and linear predictor $\widehat{X}$ with known mean $\mu=0$ (orange) and unknown mean (green)  out of ten of observations.
    The covariance structure $C$ of $X$ is chosen to be  exponential  (\ref{expo-n-10}), Gaussian (\ref{gaus-n-10}), {Bessel  (\ref{bessel-n-10}) or sine (\ref{sine-n-10}).}}}
    \label{com-small}
\end{figure}


\begin{figure}[!]
    \centering
    \begin{subfigure}{\textwidth}
    \includegraphics[height=0.25\textheight, width = \linewidth]{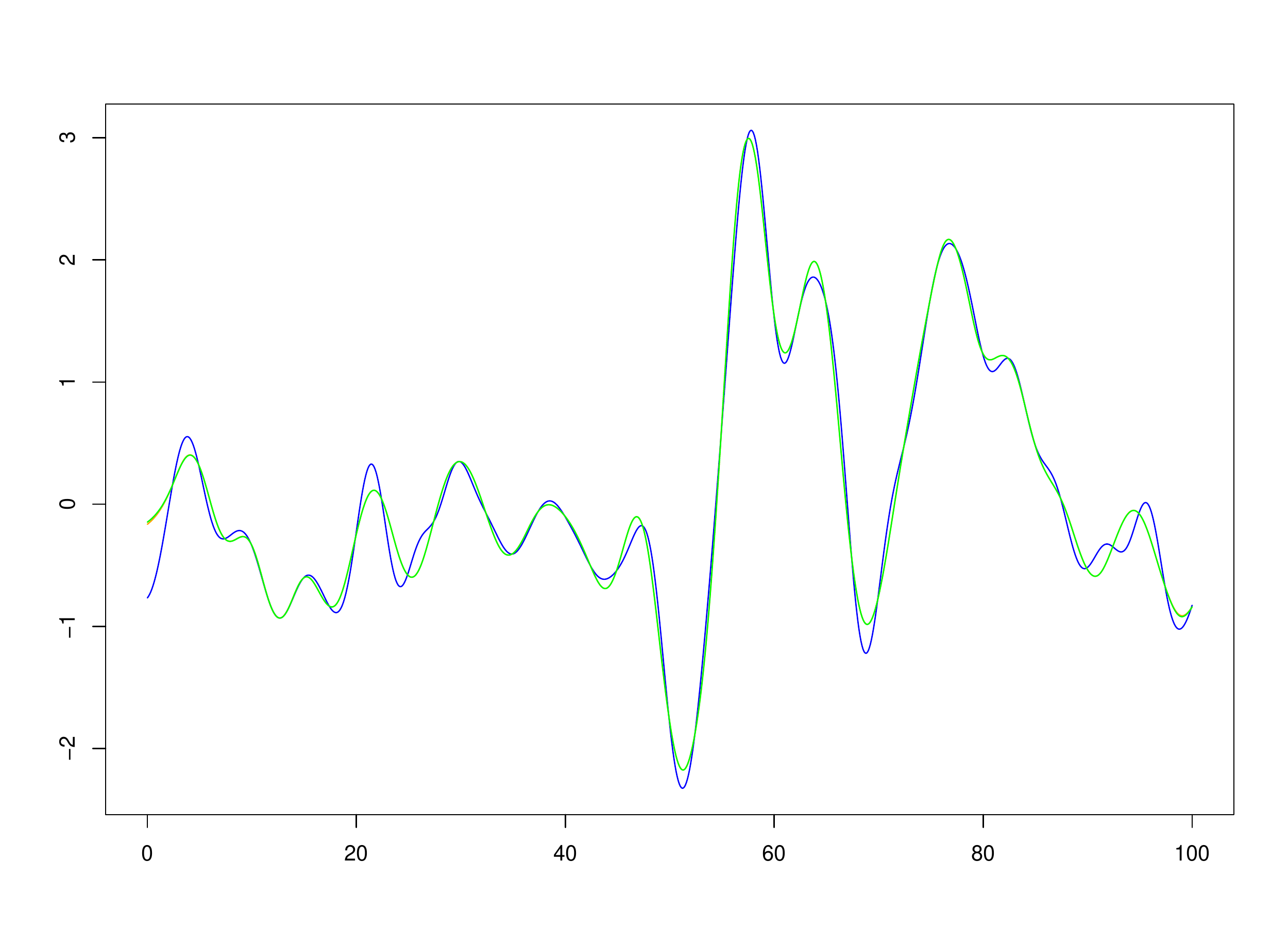}
    \caption{Gaussian covariance $C_2$, observed at $t_j= {2.5}j$, $j=1,\ldots,40$}
    \label{gauss-n-40}
    \end{subfigure}
    \newline
    \begin{subfigure}{\textwidth}
    \includegraphics[height=0.25\textheight, width = \linewidth]{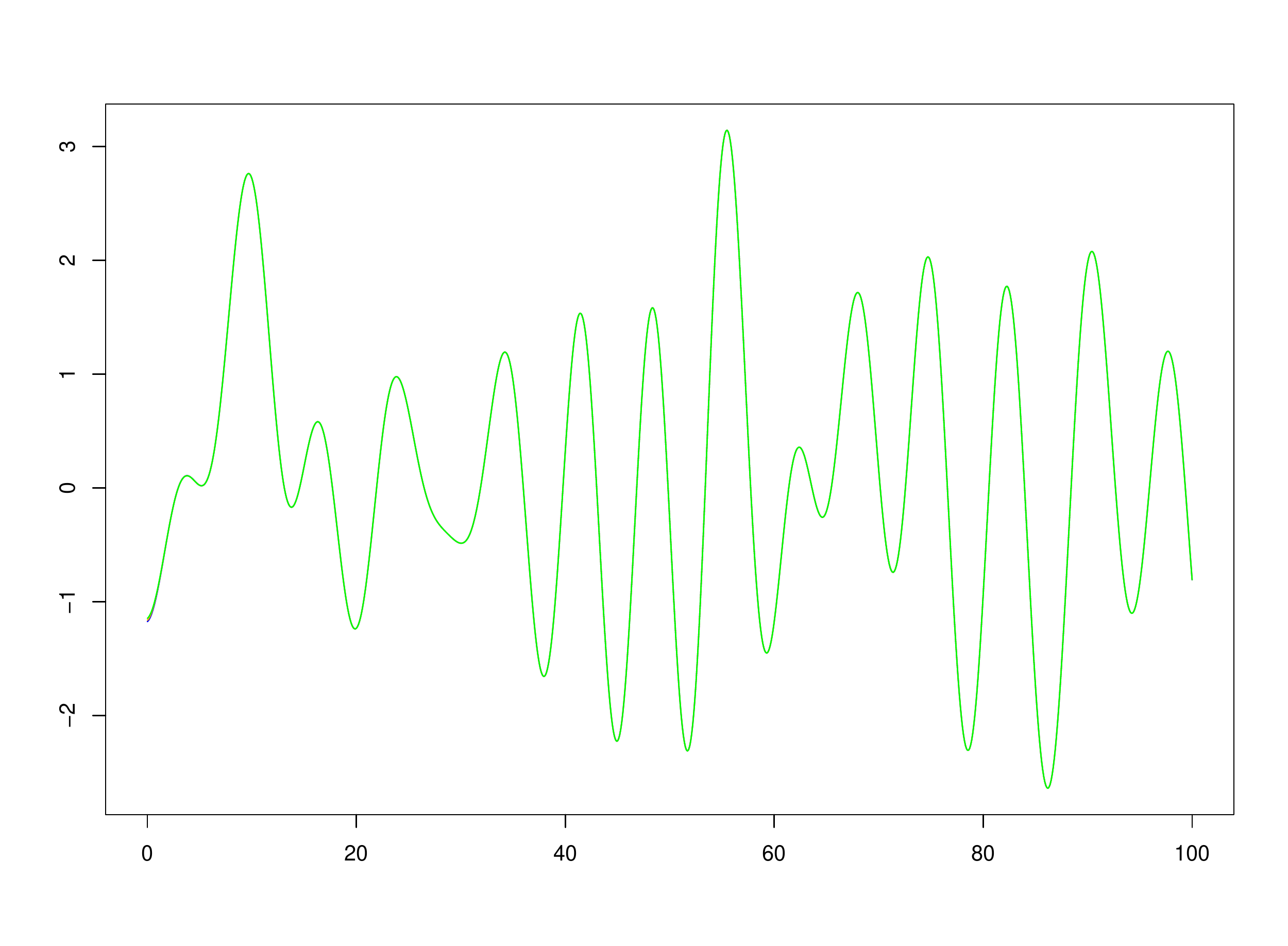}
    \caption{Bessel covariance $C_3$, observed at $t_j=2.5j$, $j=1,\ldots,40$}
    \label{bess-n-40}
    \end{subfigure}
    \newline
    \begin{subfigure}{\textwidth}
    \includegraphics[height=0.25\textheight, width = \linewidth]{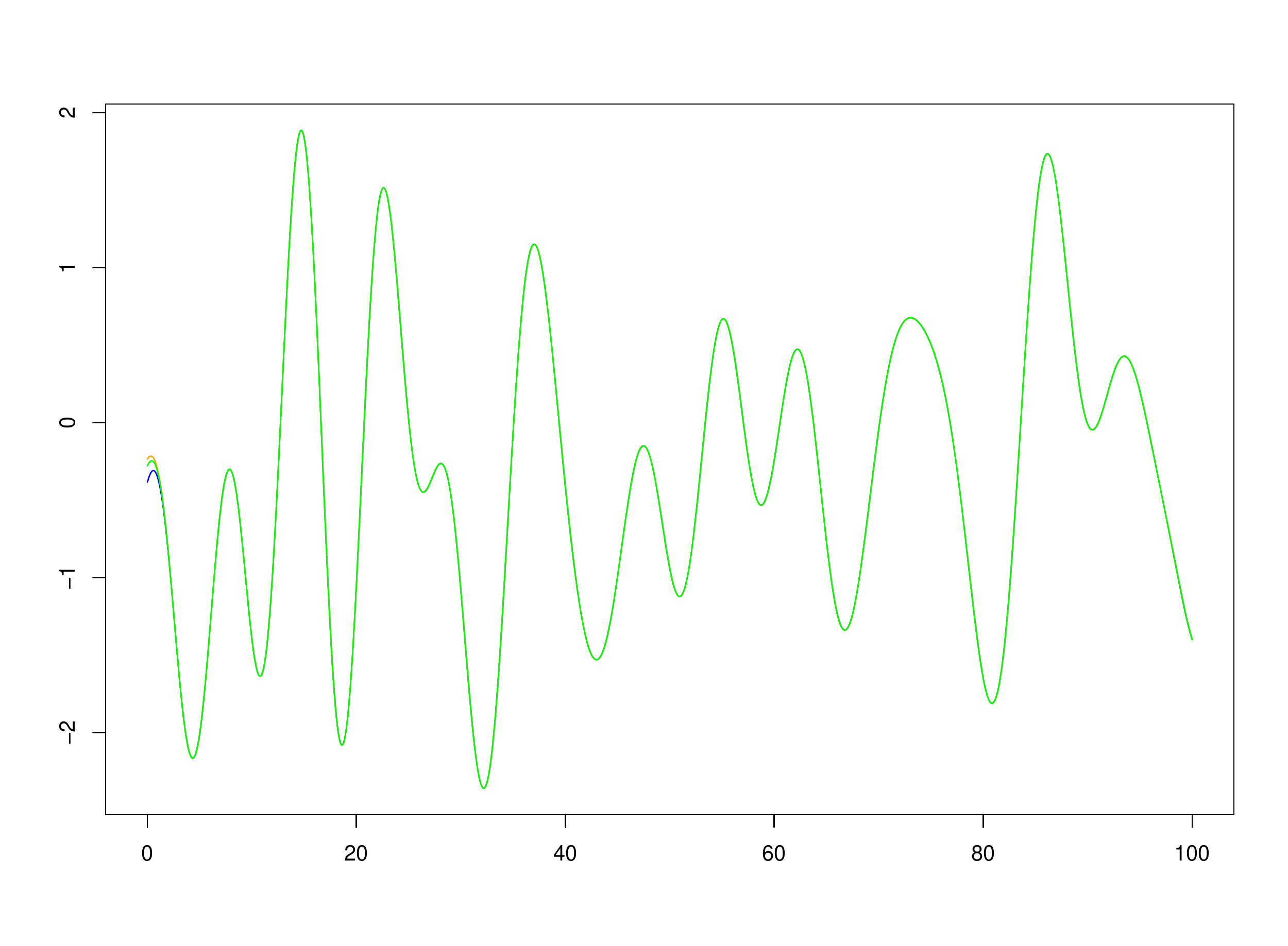}
    \caption{Sine covariance $C_4$, observed at $t_j=2.5j$, $j=1,\ldots,40$}
    \label{sin-n-40}
    \end{subfigure}
    \caption{\small{Comparison of trajectories of Gaussian process $X$ (blue) and its predictor $\widehat{X}$   with known mean $\mu=0$ (orange) or unknown mean (green) observed at $40$ locations. $X$ has Gaussian  (\ref{gauss-n-40}), {Bessel (\ref{bess-n-40}) or  sine  (\ref{sin-n-40}) covariance function.}}}
    \label{com-dense}
\end{figure}

\begin{figure}[!]
    \centering
    \begin{subfigure}{\textwidth}
    \includegraphics[height=0.22\textheight, width = \linewidth]{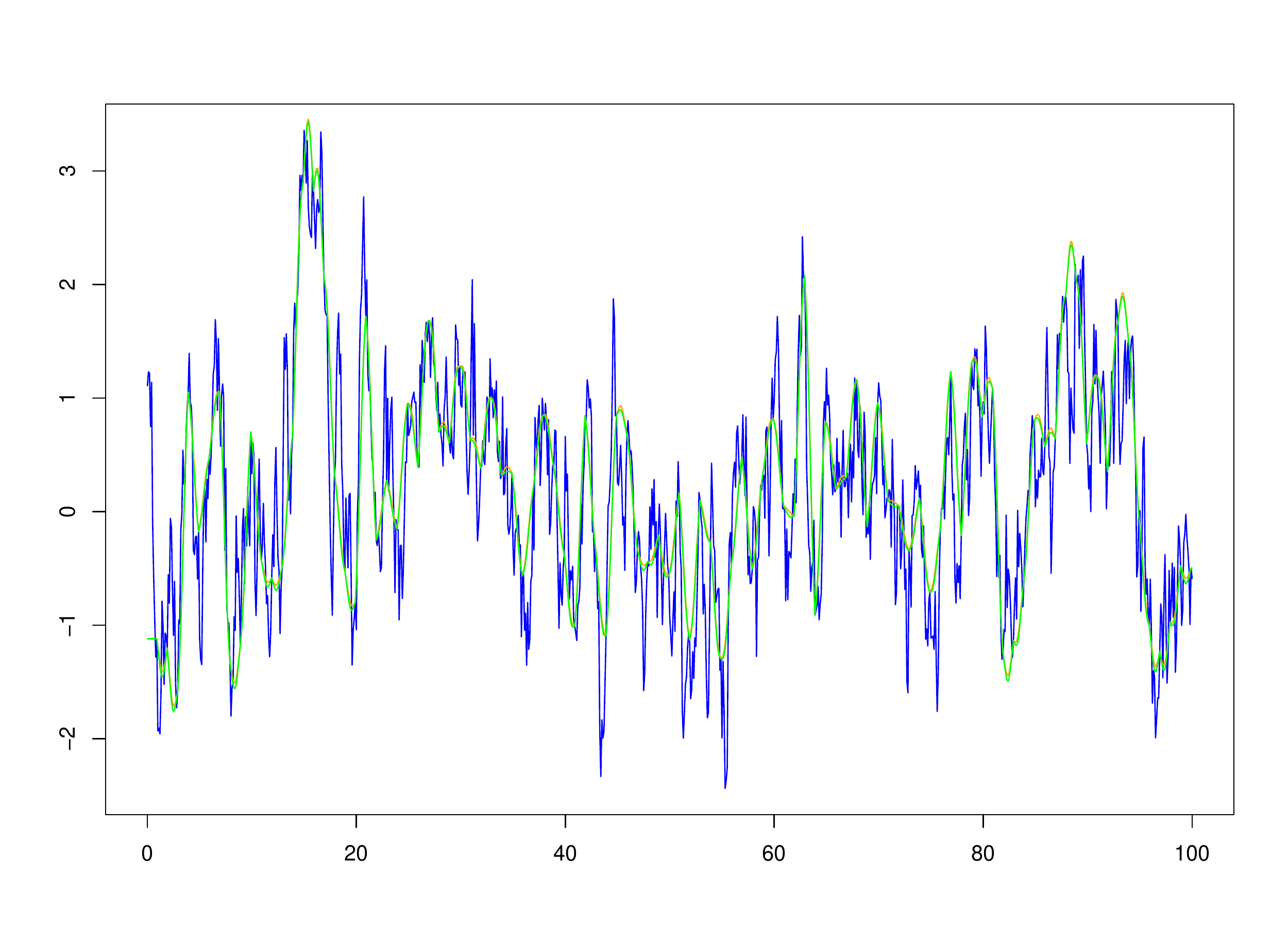}
    \caption{$X$ with exponential covariance $C_1$, observed at $t_j= j$, $j=1,\ldots,100$.}
    \label{expo-n-100}
    \end{subfigure}
    \newline
    \begin{subfigure}{\textwidth}
    \includegraphics[height=0.22\textheight, width = \linewidth]{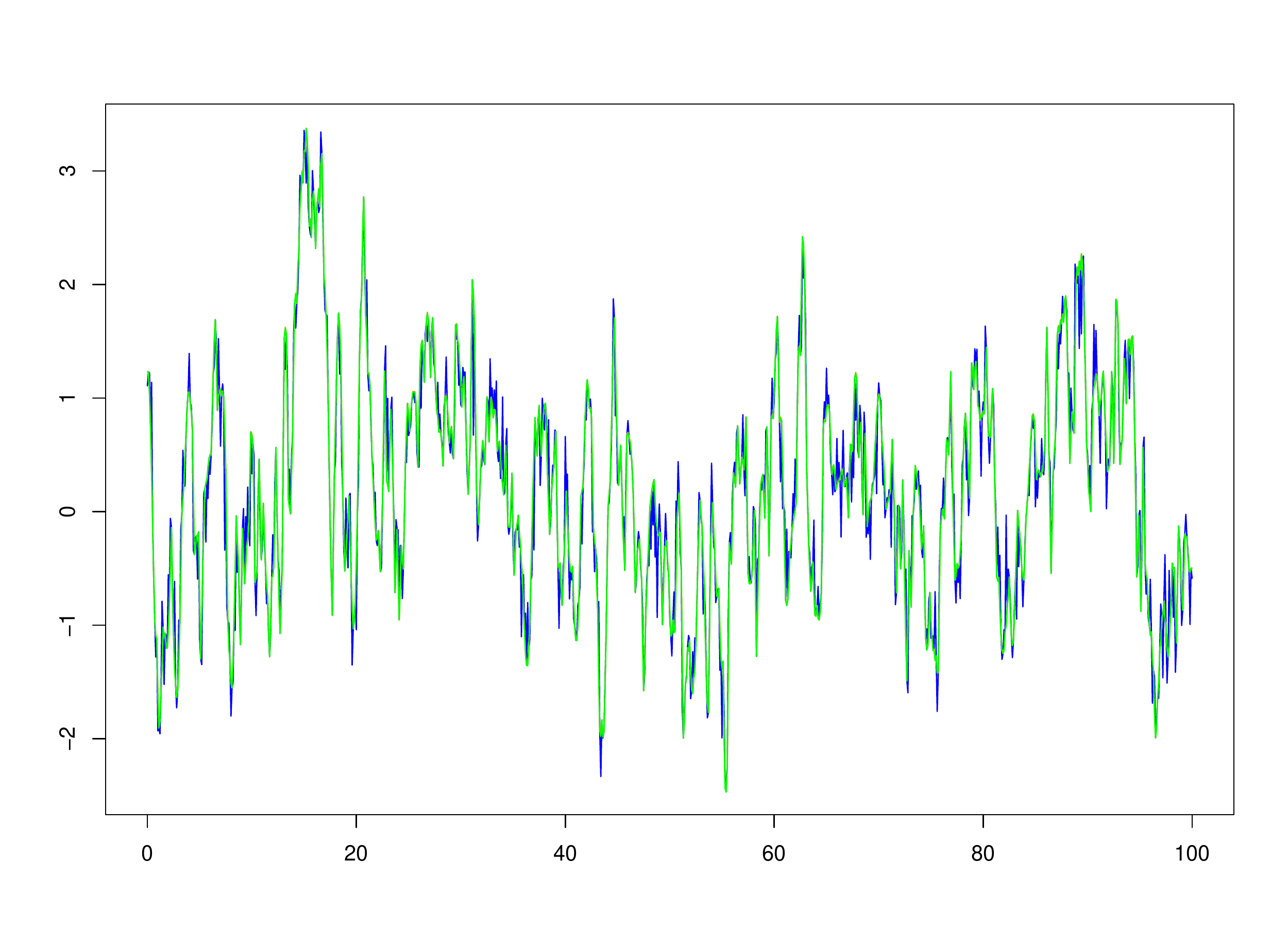}
    \caption{$X$ with exponential covariance $C_1$, observed at $t_j=0.2j$, $j=1,\ldots,500$.}
    \label{expo-n-500}
    \end{subfigure}
    \caption{{Comparison of trajectories of  Gaussian process $X$ with exponential covariance (blue) and its predictor $\widehat{X}$ with unknown mean (green)  observed at $t_j=0.2j$, $j=1,\ldots,500$ (\ref{expo-n-500}) and $t_j= j$, $j=1,\ldots,100$ (\ref{expo-n-100}). }}
    \label{expo-expo}
\end{figure}
\begin{figure}[!]
    \centering
    \begin{subfigure}{\textwidth}
    \includegraphics[height = 0.28\textheight, width = \linewidth]{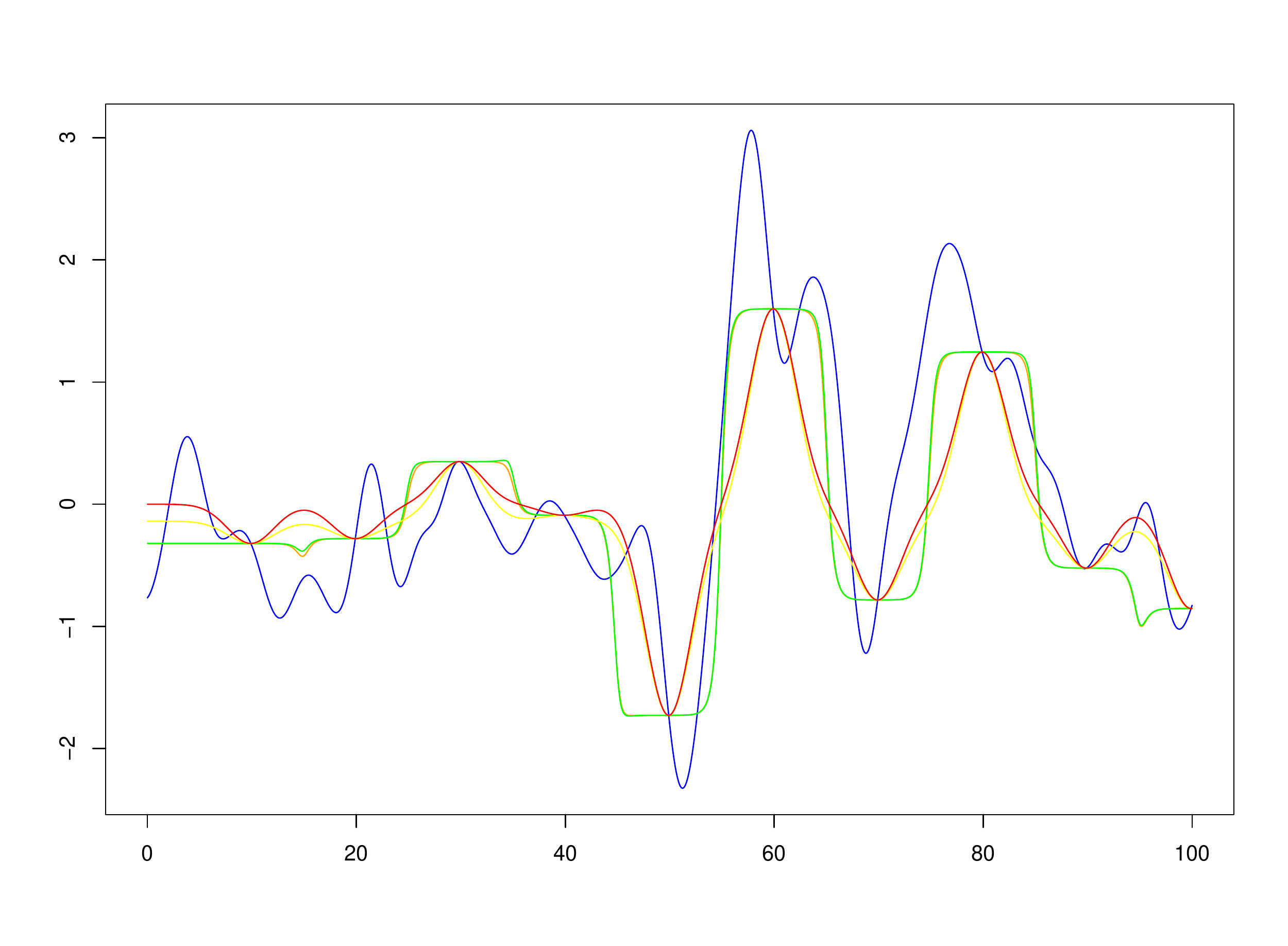}
    \end{subfigure}
    \caption{{A path of a Gaussian process $X$ (blue) with Gaussian covariance function is compared to our new linear predictor $\widehat{X}$ with known mean $\mu=0$  (orange),  with unknown mean (green), simple kriging (red) and ordinary kriging (yellow).} } \label{com-all-n-10}
\end{figure}

Next, we compare the performance of our new extrapolation method with  simple and ordinary kriging on the observation grid $t_{j}=10j,j= 1,\ldots,10$. Figure \ref{com-all-n-10} contains a realisation of a Gaussian process with Gaussian covariance $C_2$ together with our predictors $\widehat{X}$ from \eqref{a7} and  \eqref{eq:ExtrMethMu0} as well as simple and ordinary kriging estimates $\tilde{X}$ from \eqref{X:kr}. {We see that both our methods perform equally well. 
 We use boxplots to visualize the distance-in-measure error} $v_{1}\left(A_{X}(u_{j}) \Delta A_{\widehat{X}}(u_{j})\right)$ as well as $v_{1}\left(A_{X}(u_{j}) \Delta A_{\tilde{X}}(u_{j})\right)$
 at levels $u_j \in \left\{-2,-1,0,1,2\right\}$. To this end, $1200$ realisations of a Gaussian process $X$ with Gaussian covariance have been simulated. Observed at ten time spots $t_j=10j$, $j=1,\ldots,10$, they have been extrapolated using the above four methods to compare: linear predictors $\widehat{X}$ with known or unknown mean as well as (simple and ordinary) kriging extrapolators $\tilde X$. It is seen in Figure \ref{boxplot} that the error is largest at the level $u_j=\mu=0$ of the mean of $X$ and decreases with increasing $|u_j|$, as expected from formula \eqref{eq:TargetF}. The mean and the median of the error is better for the kriging methods for levels $u_j=\pm 1, \pm 2$ and comparable or slightly worse than our extrapolators for $u_j=0$. To explain this, it is enough to recall that our extrapolators  \eqref{a7} and  \eqref{eq:ExtrMethMu0} have the minimal mean distance-in-measure error in the class of all linear predictors with the same marginal distribution as $X(0)$. However, kriging methods minimize the prediction variance without this additional restriction.  
 
 {
    We extend our simulations and repeat the above scheme with $\mu=1,$ $u\in [-1,3]$ and covariance function $C_3.$ The mean values of $v_{1}\left(A_{X}(u_{j}) \Delta A_{\widehat{X}}(u_{j})\right)$ (predictor \eqref{a7}) and $v_{1}\left(A_{X}(u_{j}) \Delta A_{\widetilde{X}}(u_{j})\right)$ (ordinary kriging) are presented in Figure~\ref{Mu1l}. One can observe that our extrapolator is better around $u=\mu=1.$ At the same time, the marginal distributions of $\tilde{X}$ differ a lot from $N(1,1).$ For example, we consider each trajectory of $X,$ $\widehat{X},$ and $\tilde{X}$ as a sample and compute the corresponding sample variances $\hat{\sigma}^2$. Under the assumption of stationarity, these estimates must be close to $\sigma^2=1.$ Boxplots for $\hat{\sigma}^2$ in Figure~\ref{Mu1r} show that it is true for $\widehat{X},$ and the median value of $\hat{\sigma}^2$ for $\tilde{X}$ is 0.212 only.
 }

\begin{figure}[ht!]
    \centering
    \begin{subfigure}{0.48\textwidth}
    \includegraphics[width = \linewidth]{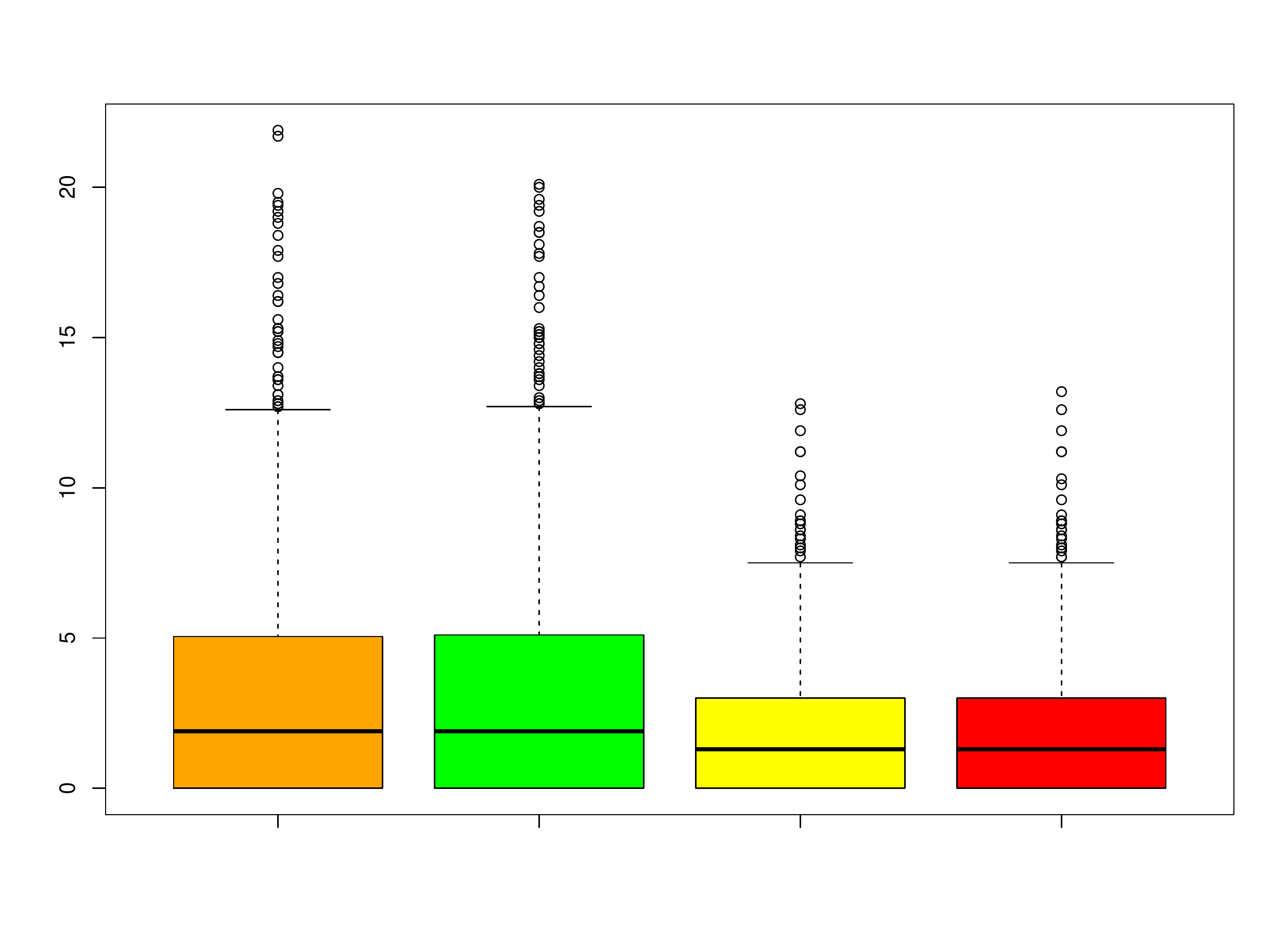}
    \caption{Distance-in-measure error of excursion sets at level $u=-2$.}
    \label{boxplot-u--2}
    \end{subfigure}
    \hfill
    \begin{subfigure}{0.48\textwidth}
    \includegraphics[width = \linewidth]{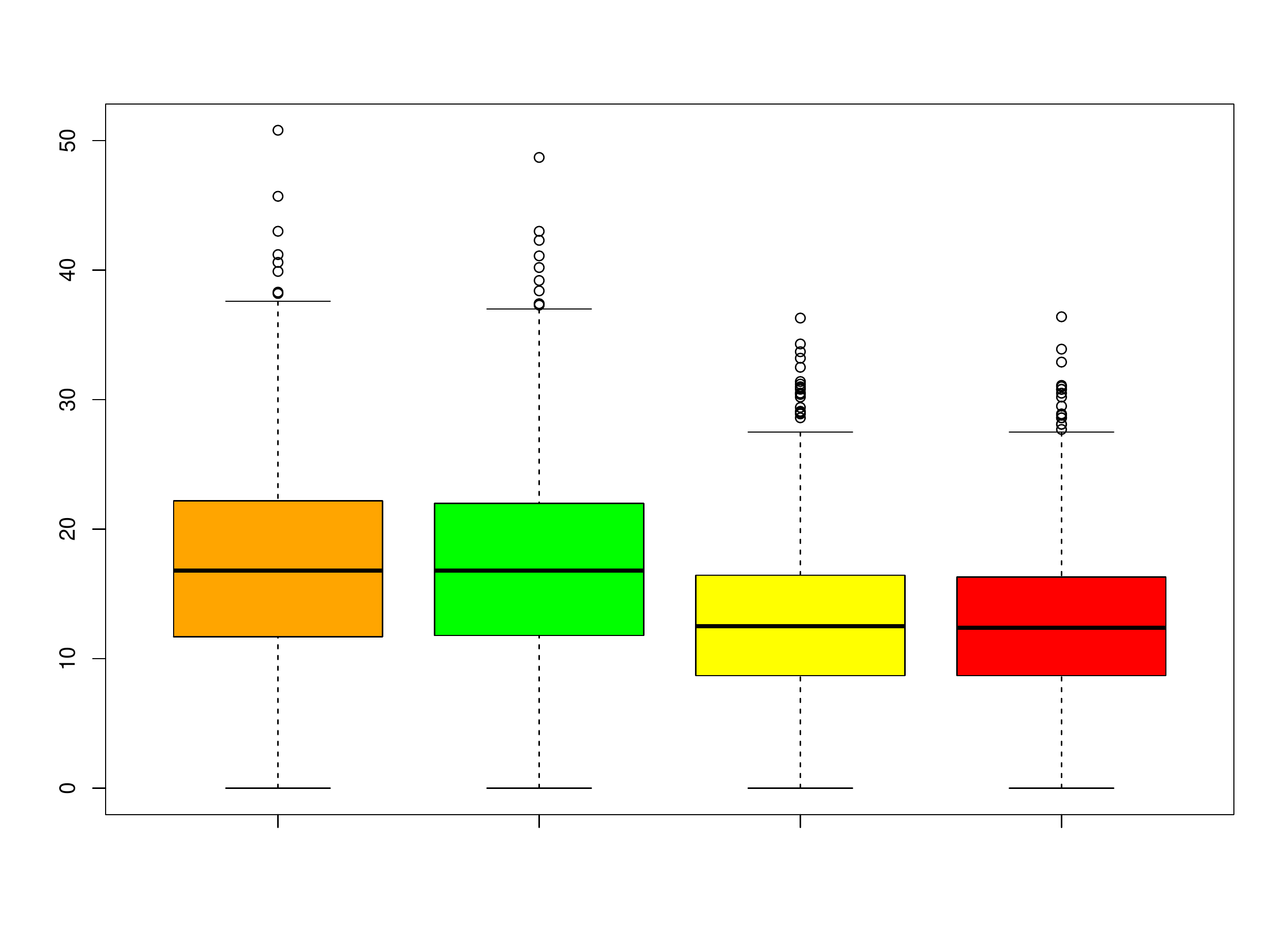}
    \caption{Distance-in-measure error of excursion sets at level $u= -1$.}
    \label{boxplot-u--1}
    \end{subfigure}
    \hfill
    \begin{subfigure}{0.48\textwidth}
    \includegraphics[width = \linewidth]{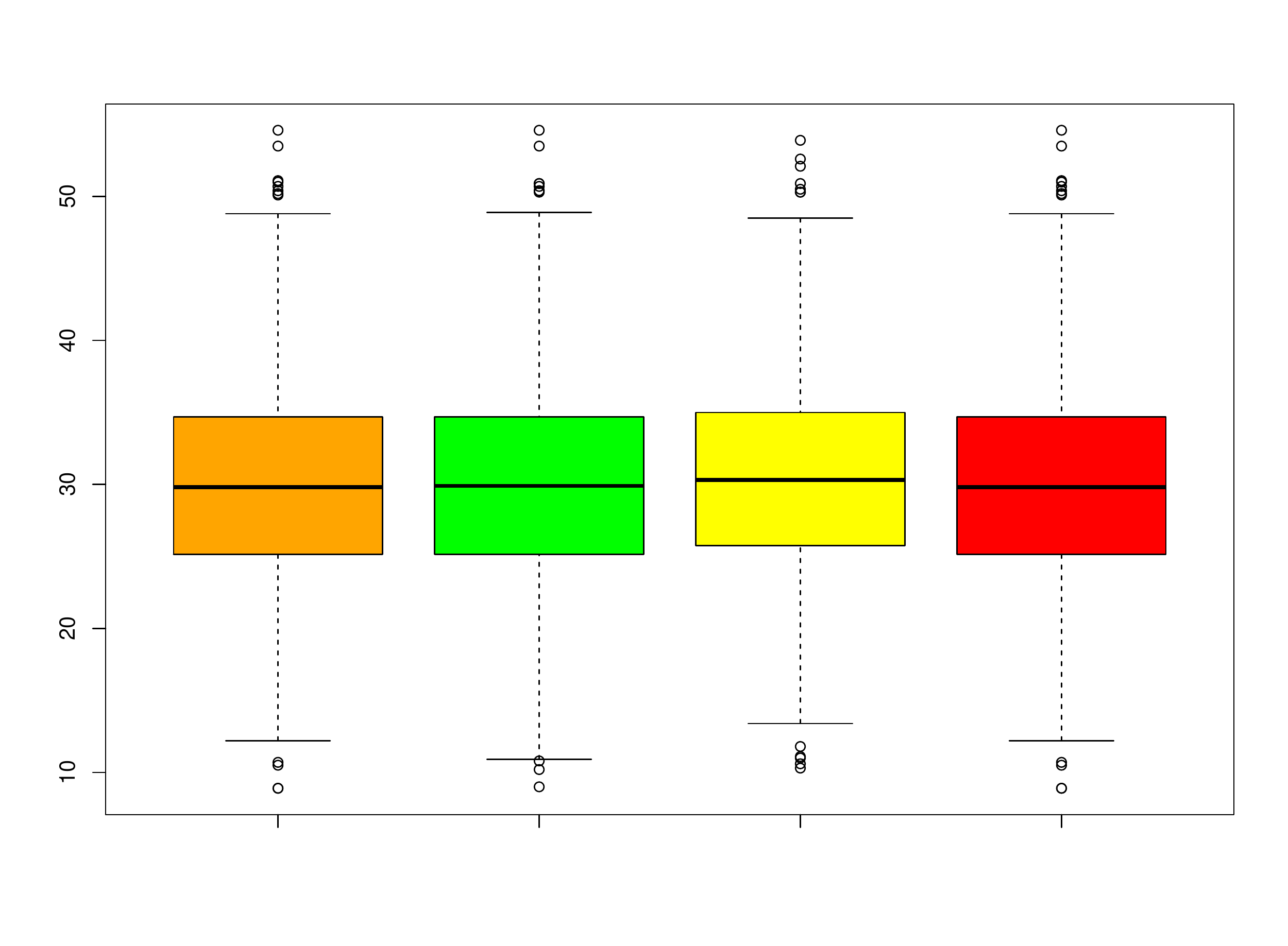}
    \caption{Distance-in-measure error of excursion sets at level $u=0$.}
    \label{boxplot-u-0}
    \end{subfigure}
    \hfill
    \begin{subfigure}{0.48\textwidth}
    \includegraphics[width = \linewidth]{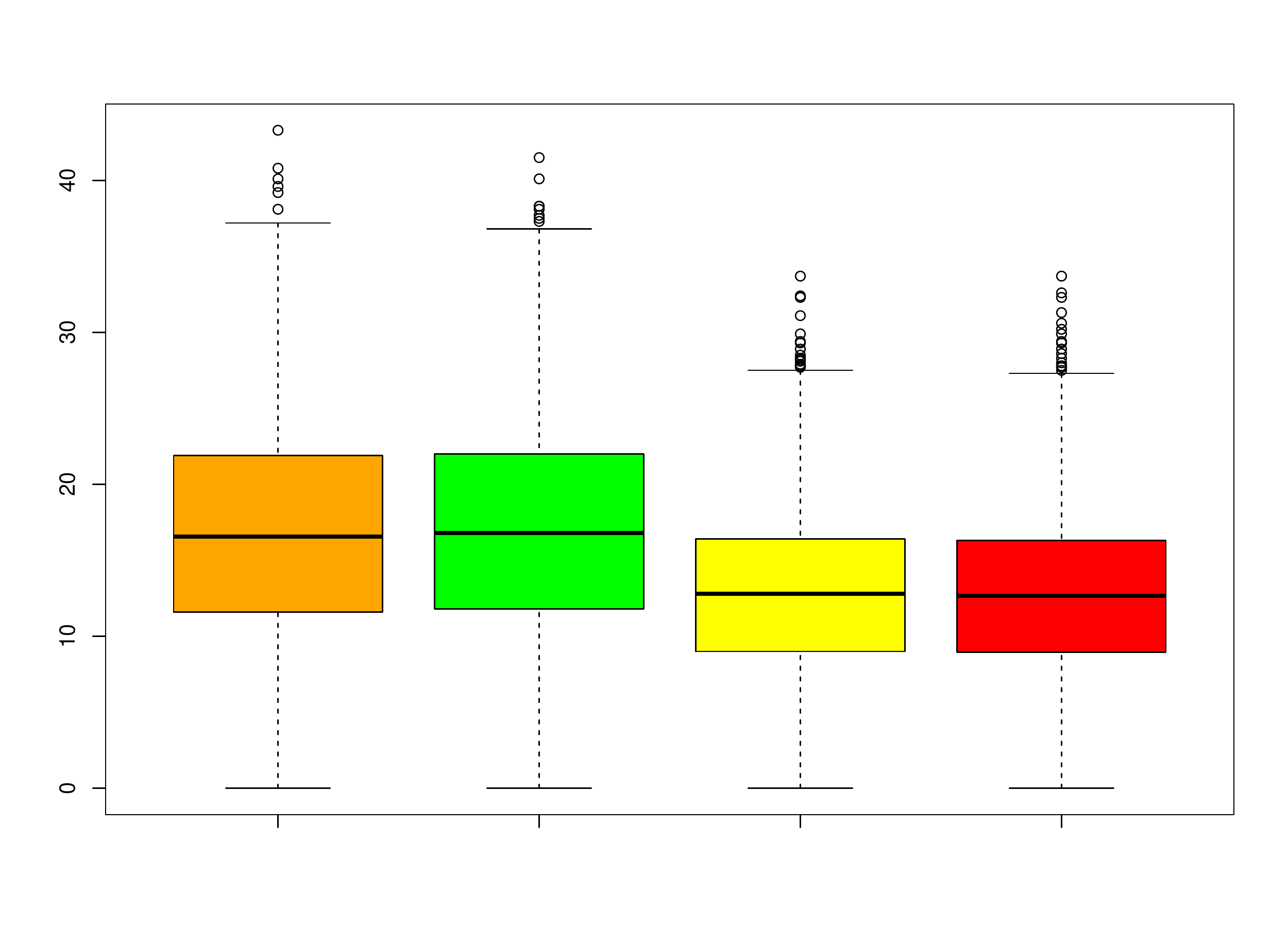}
    \caption{Distance-in-measure error of excursion sets at level  $u=1$.}
    \label{boxplot-u-1}
    \end{subfigure}
    \hfill
    \begin{subfigure}{0.48\textwidth}
    \includegraphics[height = 3.2cm,width = \linewidth]{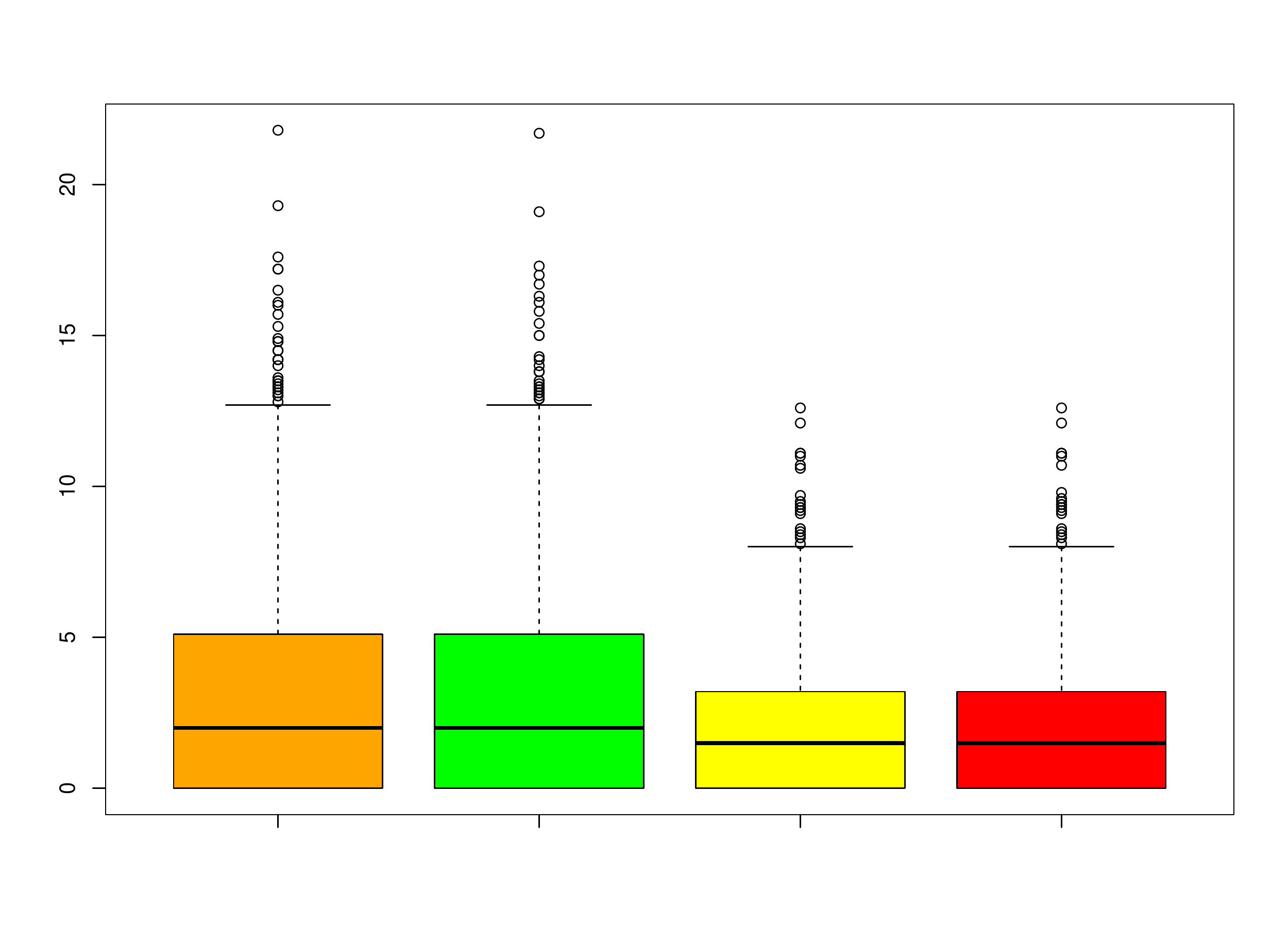}
    \caption{Distance-in-measure error of excursion sets at level  $u=2$.}
    \label{boxplot-u-2}
    \end{subfigure}
    \caption{{Boxplots for the length of the symmetric difference of excursion sets of $X$ and predictor $\widehat{X}$ with known mean $\mu=0$ (orange), unknown mean  (green) as well as  kriging predictor $\tilde X$: ordinary (yellow) and simple  (red).  The stationary Gaussian process $X$ with Gaussian covariance function is simulated 1200 times. The excursions are taken at levels (\ref{boxplot-u--2}) $u = -2$, (\ref{boxplot-u--1}) $u = -1$, (\ref{boxplot-u-0}) $u = 0$,   (\ref{boxplot-u-1})    $u = 1$  and (\ref{boxplot-u-2}) $u = 2$.}}
    \label{boxplot}
\end{figure}

\begin{figure}[ht!]
    \centering
     \begin{subfigure}{0.45\textwidth}
    \includegraphics[width = 0.95\linewidth]{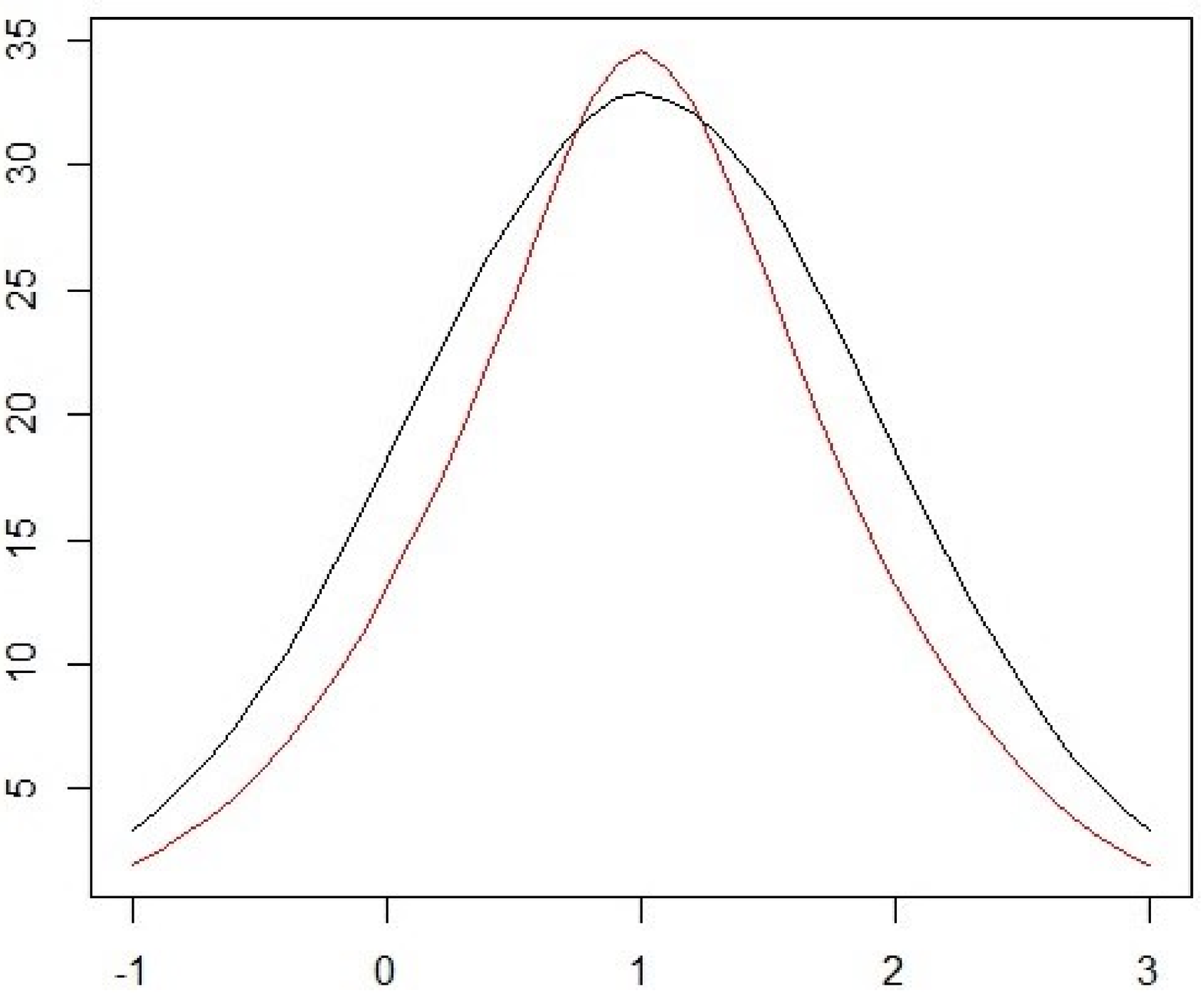}
    \caption{{The mean length of the symmetric difference of excursion sets of $X$ and predictor $\widehat{X}$ with unknown mean  (black) as well as ordinary kriging predictor $\tilde X$ (red). The excursions are taken at levels $u\in[-1,3].$}}
    \label{Mu1l}
    \end{subfigure}
    \hspace{1em}
    \begin{subfigure}{0.45\textwidth}
        \includegraphics[width = 0.95\linewidth]{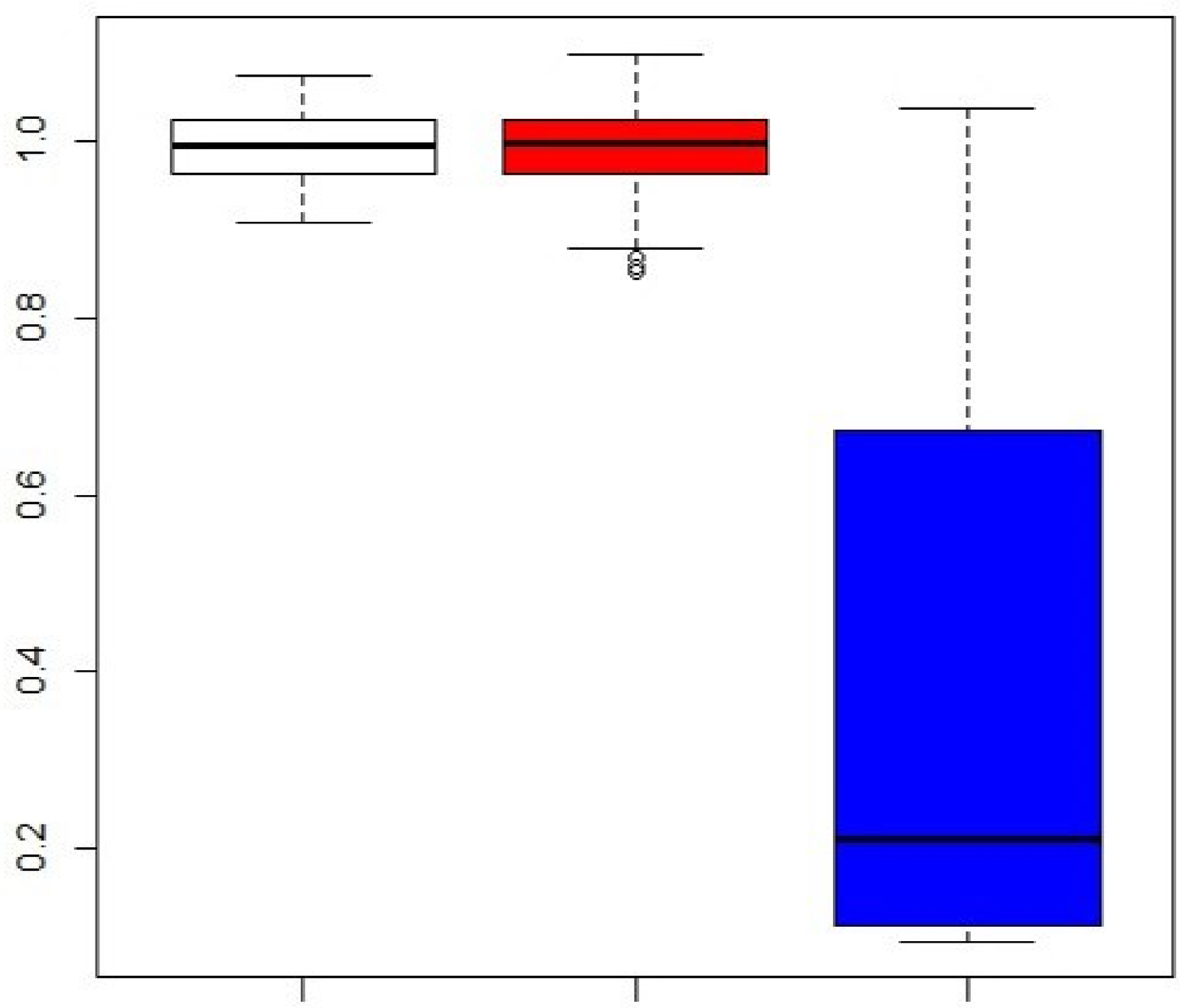}
    \caption{{Box plots of the variance estimators based on simulated trajectory $X$ (black), predictor $\widehat{X}$ with unknown mean  (red) as well as ordinary kriging predictor $\tilde X$ (blue).}}
    \label{Mu1r}
    \end{subfigure}
    \caption{{Comparison of level sets of   $\widehat{X}$ and $\tilde X$ based on   $1200$ simulated paths of the stationary Gaussian process $X$ with the Bessel covariance function ($\mu=1$, $\sigma=1$).}}
    \label{Mu1}
\end{figure}

 To summarize, we developed a quite universal framework for the linear extrapolation of stationary random functions without any additional moment existence assumptions. This framework performs well (comparable to kriging) in the Gaussian case. It is a future challenge to apply it to heavy-tailed stationary infinitely divisible  random functions.

 \bibliographystyle{amsplain}
 \bibliography{mybibfile}

\providecommand{\bysame}{\leavevmode\hbox to3em{\hrulefill}\thinspace}
\providecommand{\MR}{\relax\ifhmode\unskip\space\fi MR }
\providecommand{\MRhref}[2]{%
  \href{http://www.ams.org/mathscinet-getitem?mr=#1}{#2}
}
\providecommand{\href}[2]{#2}
\begin{thebibliography}{10}

\bibitem{AdTay07}
R.~Adler and J.~Taylor, \emph{Random fields and geometry}, Springer Monographs
  in Mathematics, Springer, New York, 2007.

\bibitem{AliGold03}
F.~Alizadeh and D.~Goldfarb, \emph{Second-order cone programming}, Math.
  Program. \textbf{95} (2003), no.~1, Ser. B, 3--51, ISMP 2000, Part 3
  (Atlanta, GA).

\bibitem{ANZH14-19}
A.~Z. Averbuch, P.~Neittaanm\"{a}ki, and V.~A. Zheludev, \emph{Spline and
  spline wavelet methods with applications to signal and image processing.
  {V}ol. {I-III}}, Springer, Cham, 2014, 2016, 2019.

\bibitem{AzWsche09}
J.-M. Aza\"{\i}s and M.~Wschebor, \emph{Level sets and extrema of random
  processes and fields}, John Wiley \& Sons, Inc., Hoboken, NJ, 2009.

\bibitem{Azzetal16}
D.~Azzimonti, J.~Bect, C.~Chevalier, and D.~Ginsbourger, \emph{Quantifying
  uncertainties on excursion sets under a {G}aussian random field prior},
  SIAM/ASA J. Uncertain. Quantif. \textbf{4} (2016), no.~1, 850--874.

\bibitem{AzzGins18}
D.~Azzimonti and D.~Ginsbourger, \emph{Estimating orthant probabilities of
  high-dimensional {G}aussian vectors with an application to set estimation},
  J. Comput. Graph. Statist. \textbf{27} (2018), no.~2, 255--267.

\bibitem{BerlTho04}
A.~Berlinet and C.~Thomas-Agnan, \emph{Reproducing kernel {H}ilbert spaces in
  probability and statistics}, Kluwer Academic Publishers, Boston, MA, 2004.

\bibitem{Bianc17}
M.~E. Biancolini, \emph{Fast radial basis functions for engineering
  applications}, Springer, Cham, 2017.

\bibitem{BolLind15}
D.~Bolin and F.~Lindgren, \emph{Excursion and contour uncertainty regions for
  latent {G}aussian models}, J. R. Stat. Soc. Ser. B. Stat. Methodol.
  \textbf{77} (2015), no.~1, 85--106.

\bibitem{BoydVan04}
S.~Boyd and L.~Vandenberghe, \emph{Convex optimization}, Cambridge University
  Press, Cambridge, 2004.

\bibitem{Chevetal14}
C.~Chevalier, D.~Ginsbourger, J.~Bect, E.~Vazquez, V.~Picheny, and Y.~Richet,
  \emph{Fast parallel kriging-based stepwise uncertainty reduction with
  application to the identification of an excursion set}, Technometrics
  \textbf{56} (2014), no.~4, 455--465.

\bibitem{chil1999geostatistics}
J.~P. Chil\'es, , and P.~D. Delfiner, \emph{Geostatistics: Modeling spatial
  uncertainty}, John Wiley \& Sons, Inc., New York, 1999.

\bibitem{Cressie93}
N.~A.~C. Cressie, \emph{Statistics for spatial data}, Wiley Series in
  Probability and Mathematical Statistics: Applied Probability and Statistics,
  John Wiley \& Sons, Inc., New York, 1993.

\bibitem{Rcode}
A.~Das, V.~Makogin, and E.~Spodarev, \emph{R code for the extrapolation of
  {G}aussian random fields with minimal error in level sets},
  \url{https://www.uni-ulm.de/fileadmin/website_uni_ulm/mawi.inst.110/mitarbeiter/spodarev/publications/Software/extrapolation_code.R},
  2021.

\bibitem{DigRib07}
P.~J. Diggle and P.~J. Ribeiro, Jr., \emph{Model-based geostatistics}, Springer
  Series in Statistics, Springer, New York, 2007.

\bibitem{GaeGu10}
C.~Gaetan and X.~Guyon, \emph{Spatial statistics and modeling}, Springer Series
  in Statistics, Springer, New York, 2010.

\bibitem{book}
A.~Genz and F.~Bretz, \emph{Computation of multivariate normal and t
  probabilities}, Lecture Notes in Statistics, vol. 195, Springer, Berlin,
  Heidelberg, 2009.

\bibitem{Holhorn13}
K.~H\"{o}llig and J.~H\"{o}rner, \emph{Approximation and modeling with
  {B}-splines}, Society for Industrial and Applied Mathematics, Philadelphia,
  PA, 2013.

\bibitem{KSS11}
W.~Karcher, E.~Shmileva, and E.~Spodarev, \emph{Extrapolation of stable random
  fields}, Journal of Multivariate Analysis \textbf{115} (2013), 516--536.

\bibitem{KletRoz04}
R.~Klette and A.~Rosenfeld, \emph{Digital geometry. geometric methods for
  digital picture analysis}, Morgan Kaufmann Publ., San Francisco; Elsevier,
  Amsterdam, 2004.

\bibitem{LaiSchu07}
M.-J. Lai and L.~L. Schumaker, \emph{Spline functions on triangulations},
  Encyclopedia of Mathematics and its Applications, vol. 110, Cambridge
  University Press, Cambridge, 2007.

\bibitem{Lantu02}
C.~Lantu\'{e}joul, \emph{Geostatistical simulation: Models and algorithms},
  Springer, Berlin, 2002.

\bibitem{Mather19}
G.~Matheron, \emph{Matheron's theory of regionalized variables}, International
  Association for Mathematical Geosciences. Studies in Mathematical
  Geosciences, Oxford University Press, Oxford, 2019, Edited by V.
  Pawlowsky-Glahn and J. Serra.

\bibitem{Moh17}
M.~Mohammadi, \emph{Prediction of {$\alpha$}-stable {GARCH} and
  {ARMA}-{GARCH}-{M} models}, J. Forecast. \textbf{36} (2017), no.~7, 859--866.

\bibitem{Moh09}
M.~Mohammadi and A.~Mohammadpour, \emph{Best linear prediction for
  {$\alpha$}-stable random processes}, Statist. Probab. Lett. \textbf{79}
  (2009), no.~21, 2266--2272.

\bibitem{Roz67}
Yu.~A. Rozanov, \emph{Stationary random processes}, Holden-Day, Inc., San
  Francisco-London-Amsterdam, 1967.

\bibitem{SamTaq94}
G.~Samorodnitsky and M.~Taqqu, \emph{Stable non-{G}aussian random processes},
  Chapman \& Hall/CRC, 1994.

\bibitem{Scheu09}
M.~Scheuerer, \emph{A comparison of models and methods for spatial
  interpolation in statistics and numerical analysis}, Ph.D. thesis,
  Georg-August Universit{\"a}t, G{\"o}ttingen, 2009.

\bibitem{ScheuSchaSchla13}
M.~Scheuerer, R.~Schaback, and M.~Schlather, \emph{Interpolation of spatial
  data---a stochastic or a deterministic problem?}, European J. Appl. Math.
  \textbf{24} (2013), no.~4, 601--629.

\bibitem{Schlaetal15}
M.~Schlather, A.~Malinowski, P.~J. Menck, M.~Oesting, and K.~Strokorb,
  \emph{Analysis, simulation and prediction of multivariate random fields with
  package randomfields}, Journal of Statistical Software \textbf{63} (2015),
  no.~8, 1--25.

\bibitem{shi96}
A.~N. Shiryaev, \emph{Probability}, Springer, New York, 1996.

\bibitem{Spodarev_2014}
E.~Spodarev, E.~Shmileva, and S.~Roth, \emph{Extrapolation of stationary random
  fields}, Stochastic Geometry, Spatial Statistics and Random Fields
  (V.~Schmidt, ed.), Springer International Publishing, aug 2014, pp.~321--368.

\bibitem{Stein99}
M.~L. Stein, \emph{Interpolation of spatial data: Some theory for kriging}, 1
  ed., Springer Series in Statistics, Springer-Verlag New York, 1999.

\bibitem{Tomita90}
H.~Tomita, \emph{Statistics and geometry of random interface systems}, World
  Scientific, 1990.

\bibitem{VazMar06}
E.~Vazquez and M.~P. Martinez, \emph{Estimation of the volume of an excursion
  set of a {G}aussian process using intrinsic kriging}, arxiv:math/0611273,
  Preprint, 2006.

\bibitem{wackernagel2013multivariate}
H.~Wackernagel, \emph{Multivariate geostatistics: An introduction with
  applications}, Springer Berlin Heidelberg, 2013.

\end{thebibliography}

\end{document}